\documentclass{amsart}

\usepackage[a4paper]{geometry}
\usepackage{amsmath,amsthm,amsfonts,amssymb,enumerate}
\usepackage{mathrsfs}
\usepackage{mathtools}
\usepackage{tikz,tikzscale}
\usetikzlibrary{decorations.markings}
\usetikzlibrary{graphs,graphs.standard,arrows.meta}
\usetikzlibrary{arrows,shapes}
\usepackage[nodisplayskipstretch]{setspace}
\usepackage{environ}
\usepackage{pgfplots,pgfplotstable,pgf}
\usepgfplotslibrary{groupplots}
\pgfplotsset{compat=newest, ticks=none}
\usepackage[final]{pdfpages}
\usepackage{import}
\usepackage{xifthen}
\usepackage{transparent}

\makeatletter
\newsavebox{\measure@tikzpicture}
\NewEnviron{scaletikzpicturetowidth}[1]{%
	\def\tikz@width{#1}%
	\def\tikzscale{1}\begin{lrbox}{\measure@tikzpicture}%
		\BODY
	\end{lrbox}%
	\pgfmathparse{#1/\wd\measure@tikzpicture}%
	\edef\tikzscale{\pgfmathresult}%
	\BODY
}
\makeatother
\tikzset{->-/.style={decoration={
			markings,
			mark=at position #1 with {\arrow{>}}},postaction={decorate}}}
\tikzset{cross/.style={cross out, draw=black, minimum size=2*(#1-\pgflinewidth), inner sep=0pt, outer sep=0pt},
	cross/.default={3pt}}

\newcommand{\bigzero}{\mbox{\normalfont\Large\bfseries 0}}

\newcommand\numberthis{\addtocounter{equation}{1}\tag{\theequation}}

\newcommand{\R}{\mathbb{R}}

\newcommand{\Q}{\mathbb{Q}}

\newcommand{\CE}{\mathcal{E}}
\newcommand{\nin}{\notin}
\newcommand{\pib}{\overline{\pi}}

\DeclareMathOperator{\Ima}{Im}

\newcommand{\sbc}{\mathscr{C}}
\newcommand{\NHIM}{\mathcal{N}}

\theoremstyle{plain}
\newtheorem{thm}{Theorem}[section]
\newtheorem{proposition}[thm]{Proposition}
\newtheorem{corollary}[thm]{Corollary}
\newtheorem{lemma}[thm]{Lemma}
\newtheorem{conjecture}[thm]{Conjecture}

\theoremstyle{definition}
\newtheorem{definition}[thm]{Definition}

\newtheorem{remark}[thm]{Remark}

\numberwithin{equation}{section}
\numberwithin{figure}{section}

\patchcmd{\subsubsection}{-.5em}{.5em}{}{}
\patchcmd{\subsection}{-.5em}{.5em}{}{}

\title[Regularisation for Simultaneous Binary Collisions]{On the $ C^{8/3} $-Regularisation of Simultaneous Binary Collisions in the Collinear 4-Body Problem}

\begin{document}
	\author{Nathan Duignan}
	\address[Nathan Duignan]{School of Mathematics and Statistics, University of Sydney, Camperdown, 2006 NSW, Australia}
	
	\author{Holger R.~Dullin}
	\address[Holger R.~Dullin]{School of Mathematics and Statistics, University of Sydney, Camperdown, 2006 NSW, Australia}
	\date{}
	
	\begin{abstract}
		The singularity at a simultaneous binary collision is explored in the collinear 4-body problem. It is known that any attempt to remove the singularity via block regularisation will result in a regularised flow that is no more than \( C^{8/3} \) differentiable with respect to initial conditions. Through a blow-up of the singularity, this loss of differentiability is investigated and a new proof of the $ C^{8/3} $ regularity is provided. In the process, it is revealed that the collision manifold consists of two manifolds of normally hyperbolic saddle singularities which are connected by a manifold of heteroclinics. By utilising recent work on transitions near such objects and their normal forms, an asymptotic series of the transition past the singularity is explicitly computed. It becomes remarkably apparent that the finite differentiability at $ 8/3 $ is due to the inability to construct a set of integrals local to the simultaneous binary collision. The finite differentiability is shown to be independent from a choice of initial condition or value of the masses. 
	\end{abstract}
	\maketitle
	\section{Introduction}

	Of central importance in the $ n $-body problem is the fact that isolated binary collisions can be regularised; a singular change of space and time variables allows trajectories to pass analytically through binary collisions unscathed. This so called Levi-Civita regularisation provides a flow smooth with respect to initial conditions. Curiously, when two binary collisions occur simultaneously, we are not so fortunate. In \cite{Martinez1999}, Martinez and Sim\'{o} gave strong evidence to conjecture the regularised flow, in a neighbourhood of the simultaneous binary collision, is at best $ C^{8/3} $. Remarkably, the conjecture was confirmed for some sub-problems of the 4-body problem \cite{Martinez2000}, including the collinear and trapezoidal problems. Despite this, the conjecture remains open for the collinear or planar $ n $-body with $ n >4 $, and the planar $ 4 $-body problem \cite{simoSomeQuestionsLooking}. 
	
	Let $ q_i(t) \in \R $ be the position of the $ i^{th} $ body on the line for $ i = 1,\dots,4 $. A \textit{simultaneous binary collision} occurs at some time $ t_c $ when two pairs of binaries, say $ (q_1,q_2) $ and $ (q_3,q_4) $, satisfy $ q_1(t_c) = q_2(t_c), q_3(t_c)=q_4(t_c) $ but $ q_1(t_c)\neq q_3(t_c) $. Throughout the paper only a spatial neighbourhood of the simultaneous binary collisions between these two \textit{distressed binaries} is considered. Denote the set of all such simultaneous binary collisions by $ \sbc $.
	
	In essence, regularisation concerns the continuation of solutions to differential systems past singular points. Solutions that approach the singularity in forward (resp. backward) time are called  ingoing (resp. outgoing) asymptotic orbits. If they can be extended through the singular point in some meaningful manner, then the singularity is deemed \textit{regularisable}. There are two primary notions of what is meant by `meaningful'. The first asks, when considered as a power series about $ t_c $ in $ t $, whether each asymptotic orbit has an analytic continuation. This is referred to as \textit{branch regularisation} or \textit{regularisation with respect to time}. It has its foundation in celestial mechanics in the work of Sundman \cite{Siegel2012} and has been considered for simultaneous binary collisions in \cite{belbrunoSimultaneousDoubleCollision1984,elbialyCollisionEjectionManifold1996,Ouyang2009,punovsevac2012regularization,saariManifoldStructureCollision1984}. 
	
	However, we are concerned with the alternate approach to regularisation whereby a singularity is regularisable if there exists an extension of the asymptotic orbits that is at least continuous with respect to initial conditions. Conley and Easton \cite{Conley1971,Easton1971} provide a precise definition of this notion, referred to as \textit{block regularisation}. They link the regularisability of a set of singularities to the behaviour of the flow in an isolating block $ N $ around the singularities. Essentially, one constructs a homeomorphism $ \pi $ by flowing ingoing points on the boundary of $ N $ to outgoing points. Note that $ \pi $ is only defined for points which are not in an asymptotic orbit. If $ \pi $ admits a unique $ C^k $ extension $ \pib $ that maps ingoing asymptotic orbits to outgoing asymptotic orbits then the set of singularities is said to be $ C^k $\emph{-regularisable} and $ \pib $ is denoted the \emph{block map}. 
	
	Many examples connecting regularity and isolating blocks are given in \cite{duignanRegularisationPlanarVector2019}. In the context of the $ n $-body problem, the ingoing and outgoing asymptotic orbits are called \textit{collision} and \textit{ejection} orbits respectively. They are denoted by $ \mathcal{E}^+ $ and $ \mathcal{E}^- $ respectively and their union is denoted by $ \mathcal{E} $.
	
	If the block map $ \pib $ is already known to be $ C^0 $, then an isolating block can be constructed from any two transverse sections $ \Sigma_0,\Sigma_3 $ of $ \mathcal{E}^+,\mathcal{E}^- $ respectively \cite{Conley1971}. It will be reproved in Theorem \ref{thm:C0Regularisable} that the set of simultaneous binary collisions $ \sbc $ is at least $ C^0 $ regularisable. Hence, the following is a simpler working definition of regularisation for the simultaneous binary collisions.
	\begin{definition}
		The set of simultaneous binary collisions $ \sbc $ is $ C^k $-regularisable if there exists two transverse sections $ \Sigma_0,\Sigma_3 $ to the collision and ejection orbits respectively and the block-map $ \pib:\Sigma_0 \to\Sigma_3 $ is $ C^k $.
	\end{definition}

	With the given definition of regularisation, Easton proved that isolated binary collisions are analytically regularisable \cite{Easton1971,Easton1972}. Yet, through the use of blow-up, it was shown by McGehee that the triple collision is not even $ C^0 $ regularisable \cite{McGehee1974}. Despite being a limiting behaviour of two isolated binary collisions, the following curious result has been conjectured.
	\begin{conjecture}[Martinez and Sim\'{o} \cite{Martinez1999} (1999)]\label{conj:C83}
		The set of simultaneous binary collisions $ \sbc $ is exactly $ C^{8/3} $-regularisable in the planar 4-body problem.
	\end{conjecture}
	Remarkably, this odd behaviour of orbits near collision has been confirmed by Martinez and Sim\'{o} \cite{Martinez2000} for some sub-problems.
	\begin{thm}[Martinez and Sim\'{o} \cite{Martinez2000} (2000)]\label{thm:martinezSimo}
		The set of simultaneous binary collisions $ \sbc $ is exactly $ C^{8/3} $-regularisable for the collinear, trapezoidal, bi-isosceles and tetrahedron 4-body problems.
	\end{thm}
	The conjecture remains open for the collinear or planar $ n $-body with $ n >4 $, and the planar $ 4 $-body problem. 
		
	There have been several authors researching work towards this conjecture. Elbialy \cite{ElBialy1990} used the blow-up method, first introduced to celestial mechanics by McGehee in \cite{McGehee1974}, to take a very general approach to the problem. Multiple collision singularities were investigated and some asymptotic behaviour of collision and ejection orbits in the $ n $-body problem was given. Elbialy's research was followed by the work of Sim\'{o} and Lacomba \cite{Simo1992} which proved the simultaneous binary collision is $ C^0 $-regularisable in the $ n $-body problem through the use of perturbative techniques. Two key papers by Elbialy, one on the collinear problem \cite{Elbialy1993collinear} and the other on the planar problem \cite{Elbialy1993planar}, produced a set of coordinates, the \textit{generalised Levi-Civita coordinates}, which showed clearly the result is $ C^0 $ in the planar problem, and further, at least $ C^1 $ in the collinear problem. It was in 1999 that Martinez and Sim\'{o} reproved the $ C^0 $ result in the plane and provided numerical evidence to support Conjecture \ref{conj:C83} for the trapezoidal $ 4 $-body problem \cite{Martinez1999}. Then, a year later, they proved Theorem \ref{thm:martinezSimo} on the $ C^{8/3} $ regularisation in some sub-problems of the planar $ 4 $-body problem. Their proof involved a type of Picard iteration to explicitly compute, order by order, some trajectories nearby collision as a function of time. After a few iterations a power of $ 8/3 $ arose in the time variable and through this the conjecture was concluded for these cases. A crucial ingredient was treating the simultaneous binary collision as a limiting case of two isolated binary collisions. In doing so, they required the time difference $ t_d $ between the two binary collisions, with simultaneous binary collision occuring when $ t_d = 0 $. However, in the planar problem, an orbit can be near collision without undergoing an isolated binary collision. As a result, there does not seem a simple extension of their method to the planar case.
	
	In this paper Theorem \ref{thm:martinezSimo} is reproved for the collinear $ 4 $-body problem. However, in an attempt to construct a method of proof that may extend to the planar problem, the more geometric path paved by Elbialy is followed. A geometrical explanation of the generalised Levi-Civita coordinates, first used by Elbialy in \cite{ElBialy1990}, is given in Section \ref{sec:coordinates}. It is shown that these coordinates regularise independent binary collisions but produce a codimension 2 set of degenerate equilibria corresponding to simultaneous binary collisions. 
	
	In the proceeding Section \ref{sec:C0regularity}, a blow-up and desingularisation produces the collision manifold. Proposition \ref{prop:structureOfCollisionManifold} is proved, revealing the collision manifold as two, 3:1 and 1:3 resonant, normally hyperbolic manifolds of singularities that are connected by a manifold of heteroclinics. A similar result was first observed in \cite{elbialy1996flow}. The proposition gives the topological structure of the flow in a neighbourhood of the set of singularities. Ultimately, this fact leads to a proof of the $ C^0 $-regularisation in Theorem \ref{thm:C0Regularisable}.
	
	Section \ref{sec:Ckregularisation} constitutes the bulk of the paper. It provides the necessary theory required to prove the main theorem, Theorem \ref{thm:C83Regularisable}, on the $ C^{8/3} $-regularisation of the simultaneous binary collisions. We begin the section by contemplating the existence of a foliation into normal, invariant 2-planes in a tubular neighbourhood of $ \sbc $. Through a study of the homological operator associated to the normal form of the set of collision singularities, the existence of the foliation is linked to the existence of a set of formal, local integrals. With this normal form procedure, a notion of how well a normal space admits a smooth, invariant foliation is defined. In particular, for the simultaneous binary collisions, a computation of the normal form in Proposition \ref{prop:normalformforSBC} concludes that the foliation fails to exist at order $ 8 $. Remarkably, the term preventing the foliation is the first term in the potential coupling the two distressed binaries. This proves a heuristic observation given by Martinez and Sim\'{o} \cite{Martinez1999} on the crucial role the coupling term plays.
	
	We continue Section \ref{sec:Ckregularisation} by noticing that the structure of the collision manifold admits a procedure for explicitly computing the asymptotic series of the block map $ \pib $. The relevant theory to compute the asymptotic orbit is detailed in \cite{duignanNormalFormsManifolds}. This theory is summarised in several propositions. It is used to prove Theorem \ref{thm:blockMapIsQuasiRegular} which asserts that the block map is generically quasi-regular. In fact, it is seen that the block map for the simultaneous binary collisions is asymptotic to a power series in terms of $ \theta^{1/3} $, where $ \theta $ will be defined as some measure of the distance away from a collision orbit. The $ 1/3 $ will be seen to result from the resonances of the normally hyperbolic singularities in the collision manifold. Finally, a cumbersome calculation involving normal forms around the normally hyperbolic manifold of singularities and solutions to variational equations is carried out to give the asymptotic series of the block map explicitly. This in turn proves the main Theorem \ref{thm:C83Regularisable}. It will become remarkably apparent that the finite differentiability at $ 8/3 $ is due to the inability to foliate the space at order $ 8 $, that is, the inability to construct a specific set of integrals local to the set of simultaneous binary collisions $ \sbc $.
	\section{Coordinates Near Simultaneous Binary Collision}\label{sec:coordinates}
	
	\subsection{Difference Vectors}
		\begin{figure}[ht]
			\centering
			\begin{scaletikzpicturetowidth}{\textwidth}
				\begin{tikzpicture}[scale = \tikzscale]
				\node[circle,fill,inner sep=2pt] (1) at (-5,0) {};
				\node[circle,fill,inner sep=1pt] (2) at (-2,0) {};
				\node[circle,fill,inner sep=2pt] (3) at (2,0) {};
				\node[circle,fill,inner sep=3pt] (4) at (4,0) {};
				\node[cross] (5) at (-4,0) {};
				\node[cross] (6) at (3,0) {}; 
				\node[inner sep=2pt,fill=white] (7) at (-4,0.5) {};
				\node[inner sep=2pt,fill=white] (8) at (3,0.5) {};
				
				\draw[-Latex] (1) -- node[label=south:$Q_1$] {}(2);
				\draw[-Latex] (3) -- node[label=south:$Q_2$] {}(4);
				\draw[-Latex] (7) -- node[auto] {$x$}(8);
				\end{tikzpicture}
			\end{scaletikzpicturetowidth}
			\caption{The configuration variables near simultaneous binary collision}\label{fig:3BDiff}
		\end{figure}
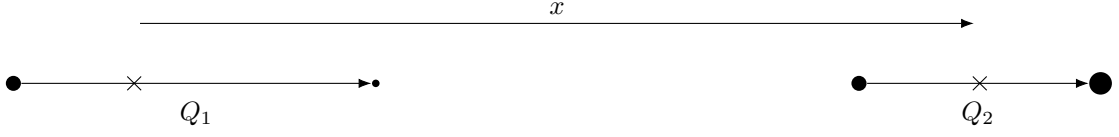
		
		Suppose there are 4 collinear bodies consisting of two binaries undergoing collision in different regions of configuration space at precisely the same time $ t_c $. Further suppose that the bodies with mass $ m_1 $ and $ m_2 $ undergo one of the binary collisions and bodies with masses $ m_3 $ and $ m_4 $ undergo the other. Let the signed distance between the bodies in each binary be given by $ Q_1,Q_2 $ respectively and let $ x $ be the signed distance between the two centre of masses of the binaries. The coordinates are depicted Figure \ref{fig:3BDiff}. If $ P_1,P_2,y $ are the conjugate momenta of $ Q_1,Q_2,x $, the dynamics of the collinear 4-body problem is given by the Hamiltonian,
		\begin{equation}
			\begin{aligned}
				H(Q,x,P,y) = \sum_{i=1}^{2}\left( \frac{1}{2 M_i} P_i^2 - k_i |Q_i|^{-1} \right) + \frac{1}{2}\mu y^2 - \hat{K}(Q_1,Q_2,x),
			\end{aligned}		
		\end{equation}
		with the standard symplectic form $ \omega = dQ_1\wedge dP_1 + dQ_2\wedge dP_2 + dx\wedge dy $.
		The smooth function $ \hat{K} $ contains the potential terms coupling the two binaries and $ M_i,k_i,\mu,d_i, c_i > 0 $ are constant functions of the masses,
		\begin{equation}
			\begin{aligned}
			\hat{K} &= \frac{d_1}{|x + c_2 Q_1 - c_4 Q_2|} + \frac{d_2}{|x + c_2 Q_1 + c_3 Q_2|} + \frac{d_3}{|x - c_1 Q_1 - c_4 Q_2|} + \frac{d_4}{|x - c_1 Q_1 + c_3 Q_2|} \\
				&~ \\
				M_1 &= \frac{m_1 m_2}{m_1 + m_2},\qquad M_2 = \frac{m_3 m_4}{m_3 + m_4},\qquad k_1 = m_1 m_2,\qquad k_2 = m_3 m_4, \\
				\mu &= \frac{m_1 + m_2 + m_3 + m_4}{(m_1 + m_2)(m_3 + m_4)},\\
				d_1 &= m_1 m_3,\quad  d_2 = m_1 m_4,\quad d_3 = m_2 m_3,\quad d_4 = m_2 m_4, \\
				c_1 &= m_2^{-1} M_1,\quad c_2 = m_1^{-1}M_1,\quad c_3 = m_4^{-1} M_2,\quad c_4 = m_3^{-1} M_2.
			\end{aligned}
		\end{equation}
		This choice of coordinates is a reduction of the system by translational symmetry via a coordinate transform that preserves the diagonal structure of the mass metric.
	
		It is more convenient to work with rescaled variables $ \tilde{Q}_i, \tilde{P}_i $ defined via the symplectic transformation
		\begin{equation}
			\tilde{Q}_i = 4 k_i M_i Q_i,\qquad \tilde{P}_i = \frac{1}{4 M_i k_i} P_i.
		\end{equation}
		Under this scaling the Hamiltonian is
		\begin{equation}
			H(\tilde{Q},x,\tilde{P},y) = \sum_{i=1}^{2} \frac{1}{2} a_i \left( \tilde{P}_i^2 - \frac{1}{2} |\tilde{Q}_i|^{-1} \right) + \frac{1}{2}\mu |y|^2 - \tilde{K}\left( \tilde{Q}, x \right),
		\end{equation}
		where $\displaystyle a_i = 16  M_i k_i^2 $ and $ \tilde{K}\left(\tilde{Q}_i,x\right) = \hat{K}\left(\frac{1}{4 k_i M_i} \tilde{Q}_i,x\right) $.
	
	\subsection{Levi-Civita Regularisation of Binaries}\label{sec:LeviCivita}
	
		In an attempt to regularise the simultaneous binary collisions, it is natural to first regularise the binary collisions of each distressed binary. This is done by passing to the Levi-Civita variables $ (\tilde{z}_i,u_i) $ through the symplectic map 
		\[ \tilde{Q}_i = \frac{1}{2}\tilde{z}_i^2,\qquad \tilde{P}_i = \tilde{z}_i^{-1} u_i. \]
		The result of this transformation is the partially regularised, translational reduced Hamiltonian,
		\begin{equation}\label{eqn:PartiallyRegularisedHam}
			H(\tilde{z},x,u,y) = \sum\limits_{i=1}^{2} \frac{1}{2} a_i \tilde{z}_i^{-2}\left(  u_i^2 - 1 \right) 
			+ \frac{1}{2}\mu y^2 - \bar{K}(\tilde{z}_1,\tilde{z}_2,x),
		\end{equation}
		with $ \bar{K}(\tilde{z}_1,\tilde{z}_2,x):= \tilde{K}\left(\frac{1}{2} \tilde{z}_1^2,\frac{1}{2}\tilde{z}_2^2,x\right) $.
		
		The Hamiltonian is said to be partially regularised for the following reason. Time can be rescaled to $ dt = z_1^2 z_2^2 d\tau $ by using the Poincar\'{e} trick of moving to extended phase space and restricting to a constant energy surface. That is, by introducing the Hamiltonian
		\begin{equation}
			\begin{aligned}
				\mathscr{H}(\tilde{z},x,u,y) &= \tilde{z}_1^2 \tilde{z}_2^2\left(H(\tilde{z},x,u,y) - h\right)\\
					&=  \frac{1}{2} a_1 \tilde{z}_2^2\left(  u_1^2 -  1\right) + \frac{1}{2} a_2 \tilde{z}_1^2\left(  u_2^2 -  1\right) 
					+ \tilde{z}_1^2 \tilde{z}_2^2\left(\frac{1}{2}\mu y^2 - \bar{K}(\tilde{z},x) - h\right),
			\end{aligned}
		\end{equation}
		and restricting to a constant energy surface in the original Hamiltonian, $ H = h $, yielding $ \mathscr{H} = 0 $. The flow on $ \mathscr{H} = 0 $ is equivalent to the flow on $ H = h $ up to time rescaling. As desired, the Hamiltonian $ \mathscr{H} $ is regular at $ \tilde{z}_1 = 0 $ or $ \tilde{z}_2 = 0 $ and so the binary collision singularities have been regularised. The set of simultaneous binary collisions $ \tilde{z}_1 = \tilde{z}_2 = 0 $, denoted by $ \sbc $, is a critical point of $ \mathscr{H} $ and the associated Hamiltonian differential equation,
		\begin{equation}\label{eqn:sBCreg}
			\begin{aligned}
				\dot{\tilde{z}}_1	&= a_1 \tilde{z}_2^2 u_1	\\
				\dot{\tilde{z}}_2	&= a_2 \tilde{z}_1^2 u_2	\\
				\dot{x}		&= \mu \tilde{z}_1^2 \tilde{z}_2^2 y
			\end{aligned}\qquad
			\begin{aligned}
				\dot{u}_1	&= \tilde{z}_1\left( 2 \tilde{z}_2^2 \left( h+\bar{K}(\tilde{z},x)-\frac{1}{2}\mu y^2 \right) - a_2 \left( u_2^2 - 1 \right) +  \tilde{z}_1 \tilde{z}_2^2\frac{\partial \bar{K}}{\partial \tilde{z}_1} \right)\\
				\dot{u}_2	&= \tilde{z}_2\left( 2 \tilde{z}_1^2 \left( h+\bar{K}(\tilde{z},x)-\frac{1}{2}\mu y^2 \right) - a_1 \left( u_1^2 - 1 \right) +  \tilde{z}_1^2 \tilde{z}_2\frac{\partial \bar{K}}{\partial \tilde{z}_2} \right)	\\
				\dot{y}		&= \tilde{z}_1^2 \tilde{z}_2^2 \frac{\partial \bar{K}}{\partial x}
			\end{aligned}
		\end{equation}
		has a manifold of singularities given by $ \sbc $. Essentially, when rescaling time to regularise the binary collisions, time was over-scaled at the set of simultaneous binary collisions, slowing down orbits as they approach the singularity and creating an equilibrium. Instead of a simultaneous binary collision occurring at some finite time $ t_c $, it now occurs as $ \tau \to \pm \infty $ for collision and ejection orbits respectively. 
		
		The following proposition gives crucial properties of the collision and ejection orbits. It has been proved in, for example, \cite{ElBialy1990,Martinez1999,Siegel2012}. We state it here in the Levi-Civita coordinates.
		\begin{proposition}\label{prop:asymptotics}
			Suppose that $ (\tilde{z}_1,\tilde{z}_2,x,u_1,u_2,y) $ is a collision (resp. ejection) orbit. Then,
			\[ u_i \to \varepsilon_i, \quad \frac{\tilde{z}_1}{\tilde{z}_2} \to \delta \left(\frac{a_1}{a_2}\right)^{\frac{1}{3}}, \quad x \to x^*, \quad y \to y^*, \]
			as $ \tau \to \infty $ (resp. $ \tau \to -\infty $). Here $ |x^*|,|y^*| < \infty $, $ \varepsilon_i = \pm 1 $ and $ \delta = \varepsilon_1 \varepsilon_2 $. Moreover, for each choice of $ \varepsilon_i $, the set $ \mathcal{E} $ of collision and ejection orbits is a 5 dimensional manifold.
		\end{proposition}
		A geometrical proof can be constructed using the methods of blow-up and desingularisation. The curious reader is referred to \cite{ElBialy1990} for details of the proof. The different values of $ \varepsilon_i $ result from the Levi-Civita transformation being a double cover of the original phase space.

	\subsection{Generalised Levi-Civita Coordinates}
		Blow-up and desingularisation methods in the Levi-Civita coordinates $ (\tilde{z}_i,u_i) $ have been implemented in \cite{Martinez1999,ElBialy1990} to produce useful asymptotic results. However, it can be argued they are not ideal coordinates to see that the set of simultaneous binary collisions is block regularisable. A set of singularities is only $ C^0 $-regularisable if each ingoing asymptotic orbit can be mapped to a unique outgoing asymptotic orbit. For the set of simultaneous binary collisions $ \sbc $, this requires that each collision orbit map to a unique ejection orbit under $ \pib $. However, from Proposition \ref{prop:asymptotics}, $ \mathcal{E}^+ $ and $ \mathcal{E}^- $ are both 5 dimensional. Consequently, if there is no obvious constraint on how orbits on $ \mathcal{E}^+ $ must map to orbits on $ \mathcal{E}^- $ then any block map $ \pi $ can not possibly be extended \textit{uniquely} to a map $ \pib $ on the whole block. 
		
		A natural constraint on how $ \mathcal{E}^+ $ maps to $ \mathcal{E}^- $ can be made by demanding that collision orbits map to ejection orbits with the same asymptotic properties. In fact, this technique is prominent in blow-up methods for algebraic geometry problems; see for instance \cite{eisenbud2006geometry}. Proposition \ref{prop:asymptotics} shows that all collision orbits approach $ \sbc $ with the same value of $ u_i^* $ and with the same tangent $ \tilde{z}_2^*/\tilde{z}_1^* $. Therefore, if it is desired to distinguish between distinct collision orbits, we must use the asymptotic value of the second derivative of the collision orbits as they approach $ \sbc $. That is, we must use the asymptotic of the ``curvature'' of each collision orbit in the $ (u_i,\tilde{z}_i) $ plane. This can be achieve by introducing a new coordinate $ \kappa_i $ through $ u_i = a_i + \kappa_i \tilde{z}_i^2 $.
		
		However, the intrinsic energy of each distressed binary is given by 
		\begin{equation}\label{eq:intrinsicEnergies}
			\tilde{h}_i = \frac{1}{2} a_i \tilde{z}_i^{-2}(u_i^2 - 1).
		\end{equation}
		Re-arranging this for $ u_i $ and expanding at the collision point $ u_i = 1,\ \tilde{z}_i = 0 $ in $ z_i $ gives \[ u_i = 1 + \frac{1}{ a_i} \tilde{h}_i \tilde{z}_i^2 + \dots . \]
		Consequently, the more physical intrinsic energies $ \tilde{h}_i $ can be used instead of the curvature $ \kappa_i $ to distinguish between distinct collision orbits. But introducing the intrinsic energies as new coordinates is precisely what is done by Elbialy in \cite{ElBialy1990}! We take a slight vairation to Elbialy by using a rescaling of 
		\begin{equation}
			z_i =  a_i^{-1/3}\tilde{z}_i,\qquad h = 2 a_i^{-1/3}\tilde{h}_i 
		\end{equation}
		to produce a version of the \textit{generalised Levi-Civita coordinates} $ (z_i,h_i,x,y) $.		
		
		The Hamiltonian in the generalised Levi-Civita coordinates and the symplectic form are, 
		\begin{equation}
			\begin{aligned}
				H &= \frac{1}{2} a_1^{1/3} h_1 + \frac{1}{2} a_2^{1/3} h_2 + \frac{1}{2}\mu y^2 - K(z_1,z_2,x), \\
				\omega &= \frac{1}{2}a_1^{1/3}\frac{ z_1^2}{ u_1} dz_1 \wedge dh_1 + \frac{1}{2}a_2^{1/3}\frac{z_2^2}{u_2} dz_2 \wedge dh_2 + dx\wedge dy.
			\end{aligned}
		\end{equation}
		where $ K(z_1,z_2,x) := \bar{K}(a_1^{1/3}z_1,a_2^{1/3}z_2,x) $.
		Of course, \eqref{eq:intrinsicEnergies} is only invertible when $ h_i z_i^2 + 1 > 0 $ and a choice of branch of $ u_i $ is made. Each of the choices will cover at least one of the simultaneous binary collision equilibria and a sufficiently small neighbourhood of $ z_1 = z_2 = 0 $ can be chosen. So, without loss of generality, make the choice $ u_i = + \sqrt{1 + h_i z_i^2} $.
		
		Using Hamilton's equations and rescaling by $ dt = z_1^2 z_2^2 d\tau $ as before, the collinear 4-body problem is given in the generalised Levi-Civita coordinates by the system
		\begin{equation}\label{eqn:generalisedLeviCivita}
			\begin{aligned}[c]
				z_1^\prime	&= z_2^2 \sqrt{1 +  z_1^2 h_1}	\\
				z_2^\prime	&= z_1^2 \sqrt{1 +  z_2^2 h_2}	\\
				x^\prime	&= \mu z_1^2 z_2^2 y			\\
			\end{aligned}
			\qquad
			\begin{aligned}[c]
				h_1^\prime	&= 2 a_1^{-1/3} z_2^2 \sqrt{1 + z_1^2 h_1} \frac{\partial K}{\partial z_1}	\\
				h_2^\prime	&= 2 a_2^{-1/3} z_1^2 \sqrt{1 + z_2^2 h_2} \frac{\partial K}{\partial z_2}	\\
				y^\prime	&= z_1^2 z_2^2\frac{\partial K}{\partial x}.
			\end{aligned}	
		\end{equation}
		Denote by $ X $ the vector field associated to System \eqref{eqn:generalisedLeviCivita}.
		
		In the generalised Levi-Civita coordinates, some properties of the flow near simultaneous binary collision become clear. Firstly, each of the binary collisions (e.g. $ z_1 = 0, z_2 \neq 0 $) are regular points of the flow, hence regularisable. There is a co-dimension 2 manifold of equilibria $ (0,0,h_1^*,h_2^*,x^*,y^*) \cong \R^4 $ corresponding to the simultaneous binary collisions $ \sbc $ in the chosen chart. Moreover, the equilibria in this manifold are degenerate in that they have vanishing Jacobian. Each fixed point in $ \sbc $ corresponds to the asymptotic values of a collision (resp. ejection) orbit as time approaches $ \infty $, (resp. $ -\infty $). That is, a fixed point in $ \sbc $ gives the value of the intrinsic energies of each distressed binary ($ h_1^*,h_2^* $), the distance between the two collisions $ x^* $, and the momentum at which the two binaries are moving from each other $ y^* $ at collision. 
		
		The fact that these equilibria are degenerate obfuscates even the topological properties of the flow in a neighbourhood of the collision set. Though, once the blow-up and desingularisation process is done in these coordinates, determining if the simultaneous binary collision is regularisable, and quantifying the degree to which it is, is a less formidable task.
		
	\section{$ C^0 $-regularity of block map}\label{sec:C0regularity}
	
	The primary aim of this section is to reprove Theorem \ref{thm:C0Regularisable} on the $ C^0 $-regularisation of simultaneous binary collisions. A didactic example, referred to as the uncoupled problem, is used to motivate the techniques and computations in this section and in Section \ref{sec:Ckregularisation}. A blow-up and desingularisation of the set of collisions $ \sbc $ in the generalised Levi-Civita coordinates and a study of the flow on the resultant collision manifold ultimately leads to the desired proof of $ C^0 $-regularisation in Theorem \ref{thm:C0Regularisable}. 
	
	\subsection{The Uncoupled Problem}
		If one uncouples the interaction between the two distressed binaries the result is the direct product of two Kepler systems. This so called \textit{uncoupled problem} is integrable. Consequently, many of the properties of the flow, such as $ C^0 $-regularisation, will follow with minimal effort. Using some of the integrals of the uncoupled problem, a lower dimensional problem can be produced and visualised. The terms in the generalised Levi-Civita system \eqref{eqn:generalisedLeviCivita} influenced by the coupling terms $ K $ are of high order in $ z_1, z_2 $. So, in the study of a tubular neighbourhood of $ \sbc $, where $ z_1 = z_2 = 0 $, removing the coupling terms should still capture the essential dynamics of the full problem.
		
		Explicitly, the uncoupled Kepler problem is given by the system,
		\begin{equation}\label{eqn:uncoupledKepler}
			\begin{aligned}
				z_1^\prime 	&= z_2^2 \sqrt{1 + z_1^2 h_1}\\
				z_2^\prime 	&= z_1^2 \sqrt{1 + z_2^2 h_2}\\
				x^\prime	&= \mu z_1^2 z_2^2 y.
			\end{aligned}
		\end{equation}
		with the other variables integrals, $ h_1^\prime=h_2^\prime=y^\prime=0 $. By making a choice of $ h_1,h_2,y $ the system can be considered as a vector field on $ \R^3 $. Similar to the coupled problem, simultaneous binary collision at $ (z_1,z_2)= (0,0) $ corresponds to a co-dimension 2 set of degenerate fixed points. Each fixed point is parameterised by $ x^* $. A qualitative plot of the dynamics is given in Figure \ref{fig:uncoupled}.
		
		\begin{figure}[ht]
			\centering
	\def\svgwidth{0.7\textwidth}
\begingroup%
  \makeatletter%
  \providecommand\color[2][]{%
    \errmessage{(Inkscape) Color is used for the text in Inkscape, but the package 'color.sty' is not loaded}%
    \renewcommand\color[2][]{}%
  }%
  \providecommand\transparent[1]{%
    \errmessage{(Inkscape) Transparency is used (non-zero) for the text in Inkscape, but the package 'transparent.sty' is not loaded}%
    \renewcommand\transparent[1]{}%
  }%
  \providecommand\rotatebox[2]{#2}%
  \newcommand*\fsize{\dimexpr\f@size pt\relax}%
  \newcommand*\lineheight[1]{\fontsize{\fsize}{#1\fsize}\selectfont}%
  \ifx\svgwidth\undefined%
    \setlength{\unitlength}{908.21708139bp}%
    \ifx\svgscale\undefined%
      \relax%
    \else%
      \setlength{\unitlength}{\unitlength * \real{\svgscale}}%
    \fi%
  \else%
    \setlength{\unitlength}{\svgwidth}%
  \fi%
  \global\let\svgwidth\undefined%
  \global\let\svgscale\undefined%
  \makeatother%
  \begin{picture}(1,0.66514929)%
    \lineheight{1}%
    \setlength\tabcolsep{0pt}%
    \put(0,0){\includegraphics[width=\unitlength,page=1]{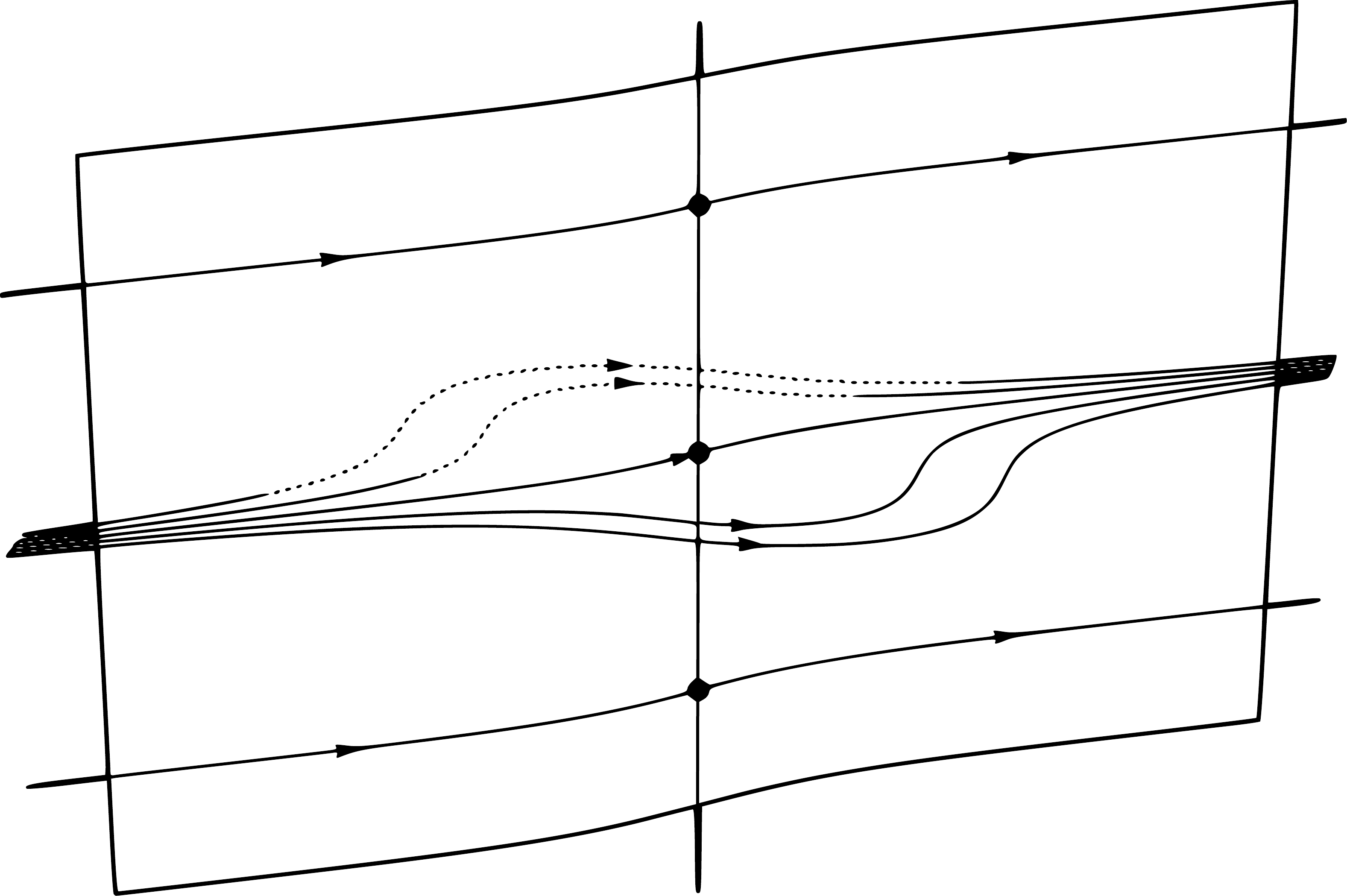}}%
    \put(0.53244235,0.03734695){\color[rgb]{0,0,0}\makebox(0,0)[lt]{\lineheight{2}\smash{\begin{tabular}[t]{l}$ \sbc $\end{tabular}}}}%
    \put(0.09359196,0.50804938){\color[rgb]{0,0,0}\makebox(0,0)[lt]{\lineheight{2}\smash{\begin{tabular}[t]{l}$ \CE^+ $\end{tabular}}}}%
    \put(0.8285484,0.60596492){\color[rgb]{0,0,0}\makebox(0,0)[lt]{\lineheight{2}\smash{\begin{tabular}[t]{l}$ \CE^- $\end{tabular}}}}%
    \put(0.52890324,0.48445527){\color[rgb]{0,0,0}\makebox(0,0)[lt]{\lineheight{2}\smash{\begin{tabular}[t]{l}$ x^* $\end{tabular}}}}%
  \end{picture}%
\endgroup%

			\caption{A qualitative plot of the uncoupled problem. Using blow-up methods in Section \ref{sec:CollisionManifold} this qualitative picture is validated. The $ C^0 $-regularity of a well chosen block map $ \pib $ is clearly apparent.}
			\label{fig:uncoupled}
		\end{figure}
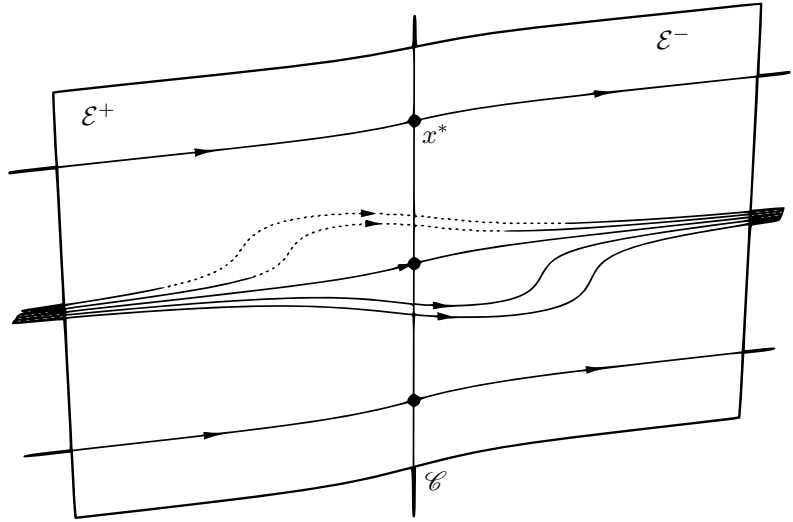
	
		From Figure \ref{fig:uncoupled}, the $ C^0 $-regularity of the block map $ \pib $ is clear. There is a manifold of collision $ \CE^+ $ (resp. ejection $ \CE^- $) orbits asymptotic to $ \sbc $ in forward (resp. backward) time. Orbits on either side $ \CE^+ $ pass around the set of singularities $ \sbc $ and meet one another on the other side near $ \CE^- $. In the next section this qualitative picture is validated for both the uncoupled and coupled problems.
		
	\subsection{Study of the Collision Manifold}\label{sec:CollisionManifold}
		In order to get the asymptotic and topological structure of the flow in a neighbourhood of $ \sbc $, blow-up and desingularisation can be performed. The use of blow-up in celestial mechanics was introduced by McGehee \cite{McGehee1974} in his study of the triple collision. Later it was implemented in investigations of the simultaneous binary collision by Elbialy \cite{ElBialy1990}, and Mart\'{i}nez and Sim\'{o} \cite{Martinez1999}. This section follows similarly to the work of Elbialy \cite{Elbialy1993planar}. We are less ambitious in our treatment of the problem in comparison to the general framework presented by Elbialy where $ l $ pairs of binaries are undergoing collision simultaneously inside the $ n $-body problem. By only treating two binaries in the $ 4 $-body problem, some simplifications and more concise statements of the flow near $ \sbc $ can be made.
		
		In the generalised Levi-Civita coordinates, the blow-up is easily achieved by introducing polar coordinates in the position variables
		\begin{equation}
			z_1 = r \cos \theta,\quad z_2 = r \sin \theta
		\end{equation}
		and the desingularisation by rescaling time $ d \bar{\tau} = r d \tau $. The result is the blown-up system,
		\begin{equation}\label{eqn:blownUpSystem}
			\begin{aligned}
				r^\prime 		&= r \sin\theta \cos\theta \left(\cos\theta \sqrt{1+ h_2 r^2 \sin^2\theta}+\sin\theta \sqrt{1 + h_1 r^2 \cos^2\theta}\right)\\
				\theta^\prime 	&= \left(\cos^3\theta \sqrt{1 + h_2 r^2 \sin^2\theta}-\sin^3 \theta \sqrt{1 +  h_1 r^2 \cos^2\theta }\right)\\
				x^\prime		&= \mu r^3 y \sin ^2\theta\cos ^2\theta \\
				h_1^\prime		&= 2 a_1^{-1/3} r \cos^2\theta \sqrt{1+ h_1 r^2 \cos^2\theta} \frac{\partial K}{\partial z_1}(r \cos\theta, r \sin\theta, x) \\
				h_2^\prime		&= 2 a_2^{-1/3}r \sin^2\theta \sqrt{1+ h_2 r^2 \sin^2\theta} \frac{\partial K}{\partial z_2}(r \cos\theta, r \sin\theta, x) \\
				y^\prime		&= r^3 \sin^2\theta \cos^2\theta \frac{\partial K}{\partial x}(r \cos\theta, r \sin\theta, x).
			\end{aligned}
		\end{equation}
		Denote the vector field associated to the system by $ X_\theta $. 
		
		In these coordinates, the set of simultaneous binary collisions corresponds to $ r = 0 $. The introduction of polar coordinates and consequent time rescaling  by $ r $ has replaced the set of equilibria occurring at $ \sbc = (0,0)\times\R^4 $ with the cylinder $ \mathcal{C} = {0}\times S^1 \times \R^4 $. The cylinder $ \mathcal{C} $ is referred to as the \textit{collision manifold}.
		
		By studying the fictitious dynamics on the collision manifold, qualitative information on nearby orbits can be gathered. The flow on the $ \mathcal{C} $ is given by setting $ r = 0 $ in \eqref{eqn:blownUpSystem},
		\begin{equation}\label{eqn:collisionFlow}
			\begin{aligned}
				r^\prime 		&= 0\\
				\theta^\prime 	&= \cos^3\theta - \sin^3 \theta \\
				x^\prime		&= h_1^\prime = h_2^\prime = y^\prime = 0.
			\end{aligned}
		\end{equation}
		Remarkably, as noted in the work of Elbialy \cite{ElBialy1990}, not only is the collision manifold invariant under the flow, but $ x,y,h_1,h_2 $ remain constant. In other words, the collision manifold $ \mathcal{C} $ is foliated by invariant $ S^1 $. The flow on these invariant circles is independent of the choice of constant $ (x^*,y^*,h_1^*,h_2^*) $. Furthermore, each $ S^1 $ has equilibria when $ \tan\theta = 1 $. This agrees with the results given in Proposition \ref{prop:asymptotics}. Combining these facts with a study the blown-up system allows us to prove the following proposition.
		
		\begin{proposition}\label{prop:structureOfCollisionManifold}
			The collision manifold $ \mathcal{C} $ is a heteroclinic connection between two normally hyperbolic invariant manifolds of fixed points. Moreover, the following properties hold:
			\begin{enumerate}[(i)]
				\item The normally hyperbolic manifolds are given by the $ (r,\theta) = (0, \pi/4) $ and $ (r,\theta) = (0,-3\pi/4) $. Denote them by $ \NHIM^-,\NHIM^+ $ respectively.
				\item The normal bundle of each manifold is 2-dimensional in the $ (r,\theta) $ directions.
				\item The heteroclinic connection is foliated by invariant $ S^1 $.
				\item Restricted to the normal bundle, the $ \NHIM^+,\NHIM^- $ are resonant hyperbolic saddles with the ratio of stable to unstable eigenvalue given by $ 1:3 $ and $ 3:1 $ respectively.
			\end{enumerate}
		\end{proposition}
		\begin{proof}
			The Jacobian of $ X_\theta $ (system \eqref{eqn:blownUpSystem}) is given by,
			\begin{equation}
				DX_\theta = 
				\left(\begin{array}{@{}cc@{}}
				\begin{matrix}
					\sin \theta \cos\theta\left( \cos\theta + \sin\theta \right) & 0 \\
					0 & -3 \sin \theta \cos\theta\left( \cos\theta + \sin\theta \right)
				\end{matrix}
				& \bigzero \\
				\bigzero & \bigzero
				\end{array}\right) + O(r).
			\end{equation}
			Evaluating on the manifolds of fixed points $ (r,\theta) = (0,\pi/4) $ and $ (r,\theta) = (0,-3\pi/4) $ yields,
			\begin{equation}
				DX_\theta|_{\NHIM^-}  = 
				\left(\begin{array}{@{}cc@{}}
				\begin{matrix}
					2^{-1/2} & 0 \\
					0 & -3\cdot 2^{-1/2} 
				\end{matrix}
				& \bigzero \\
				\bigzero & \bigzero
				\end{array}\right),\qquad DX_\theta|_{\NHIM^+}  = 
				\left(\begin{array}{@{}cc@{}}
				\begin{matrix}
					-2^{-1/2} & 0 \\
					0 & 3\cdot 2^{-1/2} 
				\end{matrix}
				& \bigzero \\
				\bigzero & \bigzero
				\end{array}\right).
			\end{equation}
			Hence, both $ \NHIM^+,\NHIM^- $ are normally hyperbolic with central directions $ (x,h_1,h_2,y) $ and each is a hyperbolic saddle with 1:3 and 3:1 resonances respectively. The unstable manifold of each fixed point in $ \NHIM^+ $ begins in the $ \theta $-direction. Due to the invariant foliation of the collision manifold into $ S^1 $, the unstable manifold must coincide with the stable manifold of a fixed point in $ \NHIM^- $ with the same values of $ (x^*,h_1^*,h_2^*,y^*) $. 
		\end{proof}
		
		Each invariant $ S^1 $ is blown-down to a single point on the manifold of simultaneous binary collisions $ \sbc $. The stable manifold of $ \NHIM^+ $ leaves the collision manifold in the $ r $ direction. Thus, the portion of the stable manifold with $ r>0 $ corresponds to $ \CE^+ $. Similarly, the portion of the unstable manifold with $ r > 0 $ of $ \NHIM^- $ is the set of ejection orbits $ \CE^- $. Because of the heteroclinic connection between the two normally hyperbolic manifolds, when the system is blown-down, $ \CE^+,\CE^- $ are glued together with each collision orbit connected to the unique ejection orbit with the same asymptotic values of $ (x^*,h_1^*,h_2^*,y^*) $. The following nice corollary can be concluded.
		\begin{corollary}
			Each collision orbit is connected to a unique ejection orbit.
		\end{corollary}
		
		Both the proposition and corollary are visually represented by a diagram of the uncoupled problem in Figure \ref{fig:uncoupledblowup}. In particular, the normal hyperbolicity of the two manifolds at $ \theta = \pi/4, -3\pi/4 $, the foliation of $ \mathcal{C} $ into invariant $ S^1 $, and that each collision orbit is uniquely connected to an ejection orbit via a heteroclinic. Note that the flow on $ \mathcal{C} $ coincides for the uncoupled and coupled problem.
		
		\begin{figure}[ht]
			\centering
	\def\svgwidth{0.7\textwidth}
\begingroup%
  \makeatletter%
  \providecommand\color[2][]{%
    \errmessage{(Inkscape) Color is used for the text in Inkscape, but the package 'color.sty' is not loaded}%
    \renewcommand\color[2][]{}%
  }%
  \providecommand\transparent[1]{%
    \errmessage{(Inkscape) Transparency is used (non-zero) for the text in Inkscape, but the package 'transparent.sty' is not loaded}%
    \renewcommand\transparent[1]{}%
  }%
  \providecommand\rotatebox[2]{#2}%
  \newcommand*\fsize{\dimexpr\f@size pt\relax}%
  \newcommand*\lineheight[1]{\fontsize{\fsize}{#1\fsize}\selectfont}%
  \ifx\svgwidth\undefined%
    \setlength{\unitlength}{884.00598241bp}%
    \ifx\svgscale\undefined%
      \relax%
    \else%
      \setlength{\unitlength}{\unitlength * \real{\svgscale}}%
    \fi%
  \else%
    \setlength{\unitlength}{\svgwidth}%
  \fi%
  \global\let\svgwidth\undefined%
  \global\let\svgscale\undefined%
  \makeatother%
  \begin{picture}(1,0.90723423)%
    \lineheight{1}%
    \setlength\tabcolsep{0pt}%
    \put(0,0){\includegraphics[width=\unitlength,page=1]{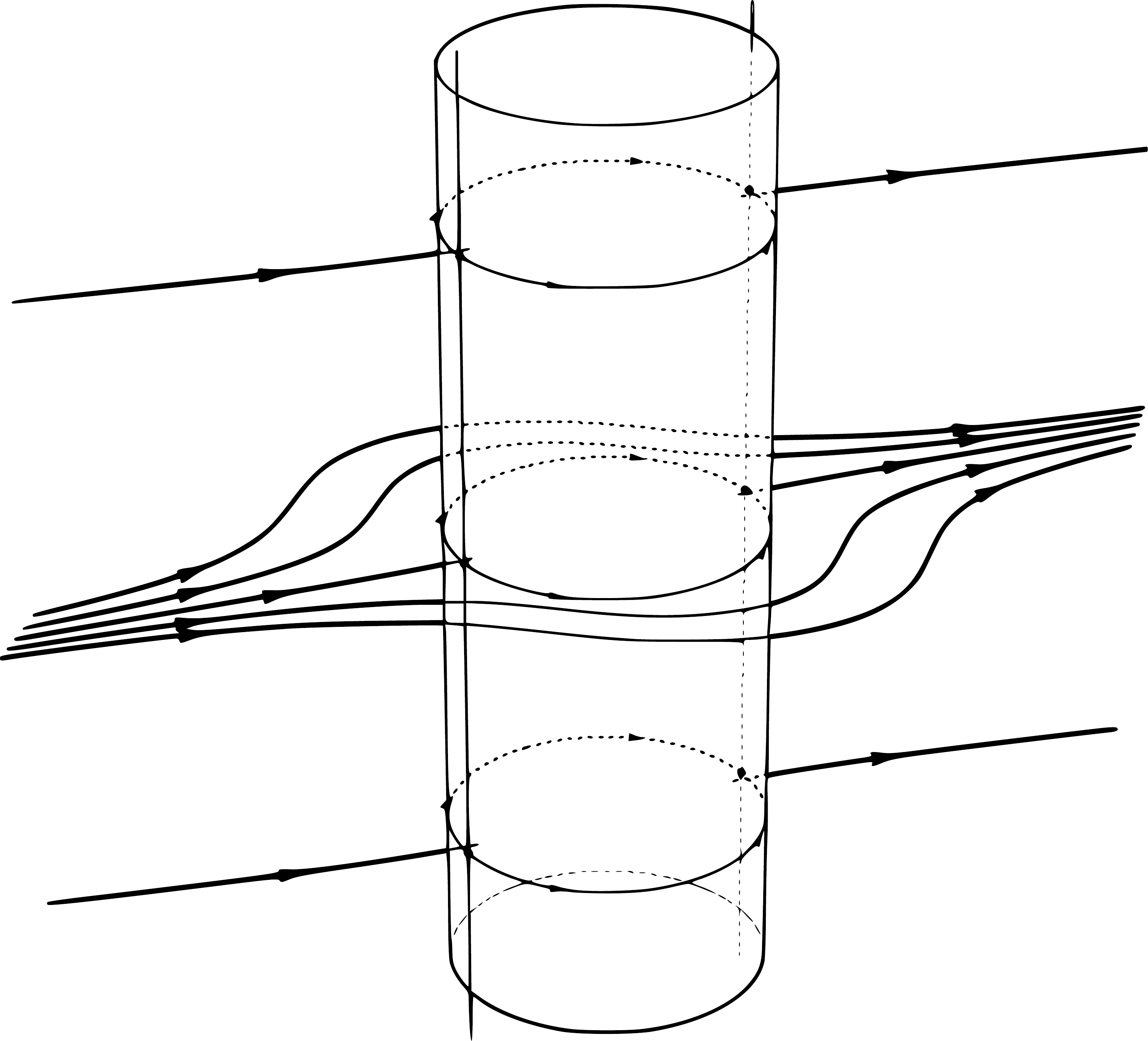}}%
    \put(0.40334123,0.84459061){\color[rgb]{0,0,0}\makebox(0,0)[lt]{\lineheight{2}\smash{\begin{tabular}[t]{l}$ \NHIM^+ $\end{tabular}}}}%
    \put(0.66150043,0.89125319){\color[rgb]{0,0,0}\makebox(0,0)[lt]{\lineheight{2}\smash{\begin{tabular}[t]{l}$ \NHIM^- $\end{tabular}}}}%
  \end{picture}%
\endgroup%

			\caption{Blow-up of the uncoupled problem. The collision manifold $ \mathcal{C} $ is represented by the cylinder and some trajectories on and nearby are given.}
			\label{fig:uncoupledblowup}
		\end{figure}

		Finally, we are in a position to prove the first key theorem, already known in \cite{Simo1992,Elbialy1993collinear,Martinez1999}.
		\begin{thm}\label{thm:C0Regularisable}
			The set of simultaneous binary collisions is at least $ C^0 $-regularisable in the collinear 4-body problem.
		\end{thm}
		\begin{proof}
			The blow-up is a diffeomorphism on $ \R^6\setminus \sbc $ and so the inverse, the blow-down, exists. The blow-down preserves the topological structure of the flow on $ (\R^+\times S^1\times\R^4) \setminus\mathcal{C} $. The $ C^0 $-regularity should now be clear. A section $ \Sigma_0 $ transverse to the manifold of collision orbits $ \CE^+ $ is split into three sets depending on whether a point is in $ \CE^+ \cap \Sigma_0 $ or which side of $ \CE^+ $ it is on. Points on the two halves of $ \Sigma_0 $ can be flowed around the collision manifold $ \mathcal{C} $ where they eventually meet again at a transverse section of the ejection orbits $ \CE^- $, say $ \Sigma_3 $. One can then glue each collision orbit in $ \CE^+\cap\Sigma_0 $ to its unique ejection orbit in $ \CE^-\cap\Sigma_3 $ to produce a $ C^0 $ block map $ \pib $.
		\end{proof}
	
	\section{$ C^{8/3} $-regularity of the Block Map}\label{sec:Ckregularisation}
	In this section Theorem \ref{thm:C83Regularisable} on the $ C^{8/3} $-regularisation of simultaneous binary collisions is proved. Firstly, a heuristic argument motivates the normal form computation to degree $ 9 $ of $ X $ in a neighbourhood of an arbitrary simultaneous binary collision on $ \sbc $. Approximate integrals are computed through a normal form procedure. An obstacle to increasing their degree is found to occur at order $ 8 $ in $ (z_1,z_2) $ in the intrinsic energy components $ h_i $ of $ X $. The non-smoothness of the block map is then established in Theorem \ref{thm:blockMapIsQuasiRegular}, where a deeper investigation of the flow near the collision manifold reveals the quasi-regularity of the block map. We give a geometric sketch of how one computes the exact regularity of the block map $ \pib $. Finally, the sketch is implemented to achieve a direct asymptotic expansion of the block map and confirm the $ C^{8/3} $-regularity. The loss of differentiability at $ 8/3 $ is linked to the failure to compute approximate integrals at order $ 8 $.
	
	\subsection{Nonlinear Normal Form Theory}
		
		The journey to proving the finite differentiability begins with a heuristic. Suppose that $ \sbc $ was indeed smoothly regularisable. If this is to mean that the singularities share properties akin to regular points, then the existence of a ``generalised flow-box'' theorem may be expected. A theorem of this nature would imply the existence of a transformation that flattens the vector field in a neighbourhood of any point in $ \sbc $. In particular, a transformation could be found that flattens the vector field in the $ (x,h_1,h_2,y) $ variables, in turn reducing the computation to a 2-dimensional problem in the $ (z_1,z_2) $ variables. If this foliation exists, not only would the number of terms in the series expansions required to compute the block map be dramatically reduced, but the theory developed in previous work on regularisation for planar vector fields \cite{duignanRegularisationPlanarVector2019} could be utilised. In other words, we want to know if there exists a foliation of a tubular neighbourhood of $ \sbc $ into invariant 2-planes normal to $ \sbc $. Of course, there could be some obstruction to this foliation and this could have an implication for the regularity of the block map. 
		
		The heuristic is supported by a study of the uncoupled problem \eqref{eqn:uncoupledKepler}. As the Kepler problem is analytically regularisable and the uncoupled problem is the direct product of two Kepler systems, the uncoupled problem must too be analytically regularisable. Moreover, the uncoupled problem admits a set of smooth integrals. These integrals give a foliation of a tubular neighbourhood of $ \sbc $ into analytic, invariant 2-planes. Hence, for the uncoupled problem, we have establishing a link between the existence of the foliation and the regularity of $ \sbc $.
		
		We turn to normal form theory to investigate the existence of a foliation in the collinear 4-body problem. Normal form theory has a long history resulting in several different normal form `styles'. For an overview see \cite{murdock2006normal}. The vector field $ X $, given in \eqref{eqn:generalisedLeviCivita}, has vanishing Jacobian on $ \sbc $, rendering the common semi-simple style useless. For this reason, the style of both Belitskii and Elphick et al \cite{belitskii1979invariant,belitskii2002c,elphickSimpleGlobalCharacterization1987}, referred to as the \textit{inner product normal form}, will be used. All styles begin the same; assume a vector field $ X $ on $ \R^n $ has a fixed point at $ 0 $ and decompose $ X $ according to some filtration. Usually this is done by taking the Taylor series of $ X $ at $ 0 $ and decomposing it into homogeneous components,
		\[ X = X_0 +X_1+\dots, \]
		with $ X_0 $ the leading order homogeneous component of degree $ s $ and $ X_{d} \in \mathcal{H}_{d+s-1} $ the space of degree $ s+d $ homogeneous vector fields. 
		
		If one applies a near identity, formal transformation of the form \[ \hat{\phi}^{-1} = I + U_d + \dots, \] where $ U_d $ is homogeneous of degree $ d+1 $, then the transformed vector field $ \hat{\phi}^*X $ at order $ d $ is given by the equation
		\begin{equation}\label{eqn:homeqn}
			(\hat{\phi}^*X)_d = X_d + [X_0, U_d],
		\end{equation}
		whilst the terms of lower order remain unchanged. Here $ [\cdot,\cdot] $ denotes the usual Lie bracket between vector fields. Because the lower terms remain unchanged, an iterative procedure on the order $ d $ can be constructed to produce a transformation putting $ X_0 $ into the `simplest' form. 
		
		Of course a choice must be made about what is meant by `simplest' form for $ \phi^*X $. Letting $ L_{X_0} := [X_0,\cdot] $, it can be seen that $ L_{X_0} $ is a linear operator acting on $ \mathcal{H}_{d} $. Hence, \eqref{eqn:homeqn} is a linear equation denoted the \textit{cohomological equation} and $ L_{X_0} $ the \textit{cohomological operator}. Denote by $ L_d:= L_{X_0}:\mathcal{H}_{d}\to\mathcal{H}_{s+d-1} $ when it is clear what $ X_0 $ is and there is a need to differentiate between which order $ L_{X_0} $ is acting. Note that in the typical case the Jacobian at the equilibrium is non-vanishing and hence the degree of of $ X_0 $ is $ s=1 $. In such a case $ L_d $ maps $ \mathcal{H}_d $ to itself. The leading order terms in the vector field \ref{eqn:generalisedLeviCivita} near $ \sbc $ are quadratic, thus we are interested in the case $ s = 2 $. 
		
		If the simplest form for $ \hat{\phi}^*X $ is deemed the one with fewest possible terms, the question then becomes, what terms of $ X_d $ can be removed by $ L_{X_0} $? Looking at the cohomological equation, terms in $ \Ima L_{X_0} \subset \mathcal{H}_{d+s-1} $ can be removed by a choice of $ U_d $. However, any terms in the complement of $ \Ima L_{X_0} $ can not be killed by the formal transformation $ \hat{\phi} $. Such terms are called \textit{resonant}.
		
		There are many choices for a space complement to $ \Ima L_{X_0} $, however, Belitskii highlights the natural choice given by $ \mathcal{H}_{d+s} =\Ima L_{X_0} \oplus \ker L_{X_0}^* $, where the adjoint is defined with respect to some choice of inner product on $ \mathcal{H}_{d+s} $. We follow Belitskii \cite{belitskii2002c} by taking the Fischer inner product on $ \mathcal{H}_d $. The following theorem combines the works of Belitskii \cite{belitskii2002c}, and Stolovitch and Lombardi \cite{lombardi2010normal}.
		\begin{thm}[\cite{belitskii2002c,lombardi2010normal}]\label{thm:FormalNF}
			There exists a formal transformation $ \hat{\phi}^{-1} = I + \sum U_d $ with $ U_d \in \Ima L_d^* $ that formally conjugates $ X = X_0+\sum X_d $ to the normal form,
			\begin{equation}
				\hat{\phi}^*X = X_0 + \sum\limits_{d\geq 1} N_d,
			\end{equation}
			with $ N_d \in \ker L_d^*  $
		\end{thm}
		Theorem \ref{thm:FormalNF} completely characterises the formal normal form for arbitrary vector fields. Moreover, it explicitly gives a way of computing both the formal transformation and the formal normal form of a given vector field $ X $. 
		
		In practice, as the computation of the normalising transformation $ \phi $ is done order by order, one only knows the normal form up to some truncated degree. Throughout the remaining sections it will become apparent that, for the determination of the $ C^{8/3} $ regularity, the vector field can be truncated at degree $ 9 $. 
		
		Let us now refocus on the problem at hand. The leading order term at any point in the singular manifold of simultaneous binary collisions is given by
		\( X_0 = (z_2^2,z_1^2,0,0,0,0) \). Letting $ w = (w_1,\dots,w_6) \in \mathcal{H}_{d+1} $ and denoting by $ \tilde{X}_0 = z_2^2\partial_{z_1}+z_1^2\partial_{z_2} $ the leading order vector field as a derivation, the adjoint of the cohomological operator is given by,
		\begin{equation}\label{eqn:cohomop}
			L^{*}_d w = \left( \tilde{X}^*_0 w_1 - 2\partial_{z_1} w_2,\ \tilde{X}^*_0 w_2 - 2 \partial_{z_2} w_1,\ \tilde{X}^*_0 w_3,\ \tilde{X}^*_0 w_4,\ \tilde{X}^*_0 w_5,\ \tilde{X}^*_0 w_6 \right), 
		\end{equation} 
		where $ \tilde{X}^*_0 = z_2 \partial_{z_1}^2 + z_1 \partial_{z_2}^2 $ is the adjoint of $ \tilde{X}_0 $. Note that $ L^*_d:\mathcal{H}_d\to\mathcal{H}_{d-1} $.
		What is immediate from the form of $ L^*_d $ is the decoupling of the $ z_1,z_2 $ components and each of the $ x,h_1,h_2,y $ components from one another. This is a consequence of the fact that $ X_0 $ decouples into the $ z_1,z_2 $ system and a trivial vector field in the other variables. The decoupling inevitably leads to a proof of Lemma \ref{lem:noFoliatation} in the next section and an answer to the question on the existence of an invariant foliation.

	\subsection{Computation of the Formal Normal Form}
		The normal form near an arbitrary simultaneous binary collision to degree $ 9 $ will now be computed. For this calculation, the Taylor series of the vector field $ X $, which is given by \eqref{eqn:generalisedLeviCivita}, around the fixed point $ (0,0,x^*,h_{1}^*,h_{2}^*,y^*) $ to degree $ 9 $ is required. Therefore, the coupled terms from the potential $ K(z_1,z_2,x) $ must be expanded to degree $ 8 $. 
		
		$ K $ is given by the sum of terms of the form,
		\[ \frac{d_i}{|x+ C_l z_1^2 + C_m z_2^2|} = \frac{d_i}{|x|} \frac{1}{|1+C_l z_1^2/x + C_m z_2^2 /x|}. \]
		Taking the sum of the four expressions of this form in $ K $ and expanding in $ (z_1,z_2) $ to degree $ 8 $ produces a series of the form $ b_0/|x| (1+ P(z_1^2/x,z_2^2/x)) $ with $ P $ some degree $ 4 $ polynomial and $ b_0 $ a function of the masses. Remarkably, the coefficients of monomials $ z_1^i z_2^j $ with degree less than $ 8 $ vanish, provided both $ i > 0, j > 0 $ (the so called `coupled' monomials). Moreover, the coefficients of $ \frac{z_1^2}{x}, \frac{z_2^2}{x} $ vanish. Thus, the desired expansion takes the form,
		\begin{equation}\label{eqn:potential}
			K(z_1,z_2,x) = \frac{1}{|x|}\left( b_{0} + K_1\left( \frac{z_1^2}{x} \right) + K_2\left( \frac{z_2^2}{x} \right) + b_{c} \frac{z_1^4 z_2^4}{x^4} \right) + \dots, \quad K_i(Q) = \sum_{j=2}^{4} b_{ij} Q^j,
		\end{equation}
		where $ b_{ij}, b_c, b_0 $ are functions of the masses given by,
		\begin{equation}
			\begin{aligned}
				b_0 	&= d_1 + d_2 + d_3 + d_4, \\
				b_{12}	&= C_1^2 (d_3 + d_4) + C_2^2 (d_1 + d_2), \qquad b_{22}	= C_4^2 (d_1 + d_3) + C_3^2 (d_2 + d_4), \\
				b_{13}	&= C_1^3 (d_3 + d_4) - C_2^3 (d_1 + d_2), \qquad b_{23}	= C_4^3 (d_1 + d_3) - C_3^3 (d_2 + d_4), \\
				b_{14}	&= C_1^4 (d_3 + d_4) + C_2^4 (d_1 + d_2), \qquad b_{24}	= C_4^4 (d_1 + d_3) + C_3^4 (d_2 + d_4), \\
				b_c		&= 6( C_1^2 (C_4^2 d_3 + C_3^2 d_4) + C_2^2 (C_4^2 d_1 + C_3^2 d_2) ), \\
				C_1		&= a_1^{1/3} \frac{1}{8 k_1 M_1} c_1, \quad C_2 =  a_1^{1/3} \frac{1}{8 k_1 M_1} c_2, \quad C_3 =  a_2^{1/3} \frac{1}{8 k_2 M_2} c_3, \quad C_4 =  a_2^{1/3} \frac{1}{8 k_2 M_2} c_4.
			\end{aligned}
		\end{equation}
		
		The first coupled monomial $ b_c z_1^4 z_2^4 $ will be seen to play a crucial role in the arrival of non-vanishing resonant terms and ultimately the finite differentiability of the block map. This proves a heuristic observation given by Martinez and Sim\'{o} \cite{Martinez1999} on the crucial role of the coupling term.
		
		The following result on the normal form near an arbitrary simultaneous binary collision can now be given. Due to the scaling symmetry of the Hamiltonian $ H $, it can be assumed that $ x $ has the asymptotic value $ x^* = 1 $.
		\begin{proposition}\label{prop:normalformforSBC}
			The normal form $ X^9 $ in a neighbourhood of the simultaneous binary collision with asymptotic values $ (x,h_1,h_2,y) = (1,h_{1}^*,h_{2}^*,y^*) $ is given to degree 9 by
			\begin{equation}\label{eqn:preBlowUpNormalForm}
			\begin{aligned}[c]
				z_1^\prime	&=  z_2^2 + (h_1 + h_1^*)^2 R_{6,1}(z_1,z_2) + (h_2 + h_2^*)^2 R_{6,2}(z_1,z_2)	\\
				z_2^\prime	&=  z_1^2 + (h_2 + h_2^*)^2 R_{6,1}(z_2,z_1) + (h_1 + h_1^*)^2 R_{6,2}(z_2,z_1)  	\\
				x^\prime	&= 0 	
			\end{aligned}
			\qquad
			\begin{aligned}[c]
				h_1^\prime	&= b_c a_1^{-1/3} R_h(z_1,z_2) 	\\
				h_2^\prime	&= b_c a_2^{-1/3} R_h(z_2,z_1) 	\\
				y^\prime	&= 0
			\end{aligned}	
			\end{equation}
			where $ R_{6,i}, R_h $  are homogeneous polynomials of degree 6 and 9 respectively. Each is given by			
			\begin{equation}
				\begin{aligned}
					R_{6,1}(z_1,z_2)	&= \frac{8}{7195}\left(-3 z_1 z_2^2 \left(20 z_1^3-13 z_2^3\right)\right), \\
					R_{6,2}(z_1,z_2) 	&= \frac{8}{7195}\left(11 z_1^6+10 z_1^3 z_2^3-10 z_2^6\right), \\
					R_{h}(z_1,z_2)		&= \frac{4}{19} (z_1-z_2) \left(z_1^2+z_1 z_2+z_2^2 \right) \left(z_1^6-11 z_1^3 z_2^3+z_2^6 \right),
				\end{aligned}
			\end{equation}		
		\end{proposition}
		\begin{proof}
			In order to check that this is indeed the normal form, the existence of a transformation taking system \eqref{eqn:generalisedLeviCivita} to system \eqref{eqn:preBlowUpNormalForm} must be computed. Then, one can check that the higher order terms in the normal form system \eqref{eqn:preBlowUpNormalForm} are elements of $ \ker L_{X_0}^* $ by simply applying $ L_{X_0}^* $ given in \eqref{eqn:cohomop} to the higher order terms. In fact, $ \tilde{X}_0^* R_h = 0 $ and this is the unique degree 9 polynomial (up to scaling) in the kernel. Due to the large number of terms in this normal form transformation it is far too unwieldy to include in this paper. The transformation can be provided upon request.
		\end{proof}
		There is a lot of information to unpack from Proposition \ref{prop:normalformforSBC}. Firstly, the normal form procedure concludes with the appearance of resonant terms at degree $ 9 $ in $ z_i $ for the $ h_i $ components of $ X $. This fact gives an answer to our question on existence of an invariant foliation of the normal space to $ \sbc $. 
		\begin{lemma}\label{lem:noFoliatation}
			It is not possible to construct analytic invariants diffeomorphic to $ h_i $.
		\end{lemma}
		\begin{proof}
			Let $ \phi_8 $ be the transformation bringing $ X $ into the normal form $ X^9 $. Note that $ \phi_8 $ is a degree $ 8 $ polynomial in $ z_1,z_2 $. Then from the fact $ x^\prime=y^\prime = 0 $ in the normal form \eqref{eqn:preBlowUpNormalForm} it is clear the $ x,y $ components of $ \phi_8 $ are invariant to order $ 8 $ in $ z_1,z_2 $. Moreover, due to the resonant terms $ R_h $ of degree $ 9 $, the $ h_i $ components of $ \phi_8 $ fail to be invariants at order $ 8 $ in $ z_i $. To prove the lemma it is sufficient to show that there does not exist a transformation of the form,
			\[ \tilde{h}_i = h_i + F(z_1,z_2,x,h_1,h_2,y), \]
			with $ F $ some analytic function such that $ \tilde{h}_i^\prime = 0 $. We will prove this for $ h_1 $ as the case $ h_2 $ follows analogously. To show this, decompose $ F $ into,
			$ F =  \sum_{d\geq 1 } F_d(z_1,z_2)  $,
			where each $ F_d $ is a homogeneous polynomial in $ (z_1,z_2) $ of degree $ d+1 $ and with coefficients analytic functions in $ x,h_1,h_2,y $. Recall that $ \tilde{X} $ is the derivation associated to a vector field $ X $ and let $ X_4^9 $ be the degree 6 components of the normal form \eqref{eqn:preBlowUpNormalForm}. Then assuming there exists $ \tilde{h}_1 $ with $ \tilde{h}_1^\prime = 0 $ we have,
			\begin{align*}
				\tilde{h}_1^\prime	&= \tilde{X}^9(\tilde{h}_1) \\
								0	&= \tilde{X}^9(h_1) + \tilde{X}^9 \left(\sum_{d\geq 1 } F_d(z_1,z_2)\right) \\
								0	&= h_1^\prime + (\tilde{X}_0 + \tilde{X}_4^9 + \dots)(F_1 + F_2 + \dots) \\
								0	&= b_c a_1^{-1/3} R_h(z_1,z_2) + \tilde{X}_0(F_1) + \tilde{X}_0(F_1) + \dots + \tilde{X}_4^9(F_1) + \dots. \numberthis\label{eqn:CorExpansion}
			\end{align*}
			Now $ \tilde{X}_0 : \mathcal{H}_d \to \mathcal{H}_{d+1} $, $ \tilde{X}_4^9:\mathcal{H}_d \to \mathcal{H}_{d+5} $ and $  R_h(z_1,z_2) \in \mathcal{H}_{8} $. Taking all elements on the rhs in $ \mathcal{H}_{8} $ we obtain the equation,
			\[ 0 = b_c a_1^{-1/3} R_h(z_1,z_2) + \tilde{X}_0(F_7) + \tilde{X}^9_4(F_3). \]
			So, we require $ F_7 $ or $ F_3 $ to be found which cancels with the $ R_h $ term if the approximate integral $ \tilde{h}_i $ exists. But, by the normal form procedure $ R_h \in \ker L_d^* $ which, from \eqref{eqn:cohomop} implies $ R_h \in \ker \tilde{X}_0^* $. As $ \Ima \tilde{X}_0 $ is the orthogonal complement to $ \ker \tilde{X}_0^* $ we are guaranteed that $ R_h \nin \Ima \tilde{X}_0  $, thus no such $ F_7 $ can be found. Moreover, by collecting the terms in $ \mathcal{H}_{3} $ terms in the expansion \eqref{eqn:CorExpansion}, we obtain $ \tilde{X}_0(F_3) = 0 $. That is, $ F_3 \in \ker \tilde{X}_0 $. Dynamically this says that $ F_3 $ is invariant under the flow of the leading order terms $ X_0 $. But it is easily verified that the only invariants of $ X_0 $ are polynomials in the leading order invariant 
			\[ \hat{\kappa} = z_1^3 - z_2^3.  \]
			As any homogeneous polynomial in $ \hat{\kappa} $ must be degree $ 3 j $ and hence in $ \mathcal{H}_{3j-1} $, it follows that $ F_3 = 0 $. Thus, no such $ F_7 $ or $ F_3 $ exists and the lemma follows in consequence.
		\end{proof}

		\begin{remark}
			Note that, after blow-up, degree $ 9 $ terms become degree $ 8 $ terms due to the rescaling by $ d\tau = r dt $. Therefore, an obstacle to the foliation is occurring at degree 8 in the blow-up space. 
		\end{remark}
		\begin{remark}\label{rmk:canIgnoreR6}
			 The resonant term $ R_h $ appearing in the intrinsic energy components have $ b_c $ as a factor, implying these terms come from the first coupled monomial $ z_1^4 z_2^4 $ from the expanded potential $ K $ in \eqref{eqn:potential}. Moreover, the absence of $ b_{ij}, b_0 $ in the normal form show $ K_1, K_2 $ make no contribution to the resonant terms in the normal form. Lastly, the only terms in $ X $ from \eqref{eqn:generalisedLeviCivita} which are independent of the mass constants are the terms in the $ z_i $ components. As $ R_{6,j} $ are independent of the masses as well, they must then come from the kinetic energy terms. A Hamiltonian with only kinetic terms has no singularities and is thus analytically regularisable. Under this reasoning, it must follow that the $ R_{6,j} $ terms make no contribution to any finite differentiability of the block map $ \pib $. This will be explicitly confirmed in the computation of the asymptotic series of $ \pib $ in Section \ref{sec:proofOfMainTheorem}.
		\end{remark}
		\begin{remark}
			There are three invariants up to order $ 8 $ given by
			\begin{equation}
				H = a_2^{-1/3} h_1 + a_1^{-1/3} h_2
			\end{equation}
			and the transformed $ x $ and $ y $ variables. There is in fact another. The degree 3 polynomial $ \hat{\kappa} = z_2^3 - z_1^3 $, an integral of $ X_0 $, was essential in the proof of Lemma \ref{lem:noFoliatation}. This integral can be extended to a degree $ 7 $ integral of the normal form $ X^9 $,
			\begin{equation}\label{eqn:kappaOriginal}
				\begin{aligned}
					\kappa(z_1,z_2,h_1,h_2) &= \frac{1}{6}(z_1^3 - z_2^3) + (h_1 + h_1^*)^2 \kappa_7(z_1,z_2) - (h_2 + h_2^*)^2\kappa_7(z_2,z_1), \\
					\kappa_7(z_1,z_2)		&= \frac{1}{50365} z_1 \left( 485 z_1^6 - 665 z_1^3 z_2^3 + 308 z_2^6 \right).
				\end{aligned}
			\end{equation}
			In a sense, the existence of this integral reflects Remark \ref{rmk:canIgnoreR6} and the heuristic argument that smooth block-regularisation may imply a type of flow box theorem. The integral $ \kappa $ will play a central role in showing the $ R_6 $ terms do not affect the $ 8/3 $ regularity of the block map.
		\end{remark}
	
	\subsection{Geometric Sketch of Proof}\label{sec:GeometricSketch}
		A procedure for determining the finite differentiability of the block map $ \pib: n^+ \to n^- $ is now sketched. Recall from Proposition \ref{prop:structureOfCollisionManifold} the topological structure of the flow near $ \mathcal{C} $. If $ \mathcal{T} $ is a tubular neighbourhood  of $ \mathcal{C} $, then Proposition \ref{prop:structureOfCollisionManifold} reveals a natural decomposition of $ \mathcal{T} $ into four overlapping neighbourhoods \[ \mathcal{T} = \mathcal{U}^+\cup\mathcal{U}^-\cup \mathcal{V}^+ \cup\mathcal{V}^- ,\] where $ \mathcal{U}^+, \mathcal{U}^- $ are tubular neighbourhoods of the normally hyperbolic invariant manifolds $ \NHIM^+, \NHIM^- $ respectively, and $ \mathcal{V}^+, \mathcal{V}^- $ are a tubular neighbourhoods of one of the two manifolds of heteroclinic orbits. The decomposition splits $ \mathcal{T} $ into regions where the flow is topologically equivalent to a neighbourhood of a normally hyperbolic saddle ($ \mathcal{U}^+, \mathcal{U}^- $) and regions where the flow is topologically equivalent to a regular flow ($ \mathcal{V}^+, \mathcal{V}^- $).
		
		\begin{figure}[h]\label{fig:uncoupledsections}
			\centering
	\def\svgwidth{0.7\textwidth}
\begingroup%
  \makeatletter%
  \providecommand\color[2][]{%
    \errmessage{(Inkscape) Color is used for the text in Inkscape, but the package 'color.sty' is not loaded}%
    \renewcommand\color[2][]{}%
  }%
  \providecommand\transparent[1]{%
    \errmessage{(Inkscape) Transparency is used (non-zero) for the text in Inkscape, but the package 'transparent.sty' is not loaded}%
    \renewcommand\transparent[1]{}%
  }%
  \providecommand\rotatebox[2]{#2}%
  \newcommand*\fsize{\dimexpr\f@size pt\relax}%
  \newcommand*\lineheight[1]{\fontsize{\fsize}{#1\fsize}\selectfont}%
  \ifx\svgwidth\undefined%
    \setlength{\unitlength}{884.00001346bp}%
    \ifx\svgscale\undefined%
      \relax%
    \else%
      \setlength{\unitlength}{\unitlength * \real{\svgscale}}%
    \fi%
  \else%
    \setlength{\unitlength}{\svgwidth}%
  \fi%
  \global\let\svgwidth\undefined%
  \global\let\svgscale\undefined%
  \makeatother%
  \begin{picture}(1,0.90905464)%
    \lineheight{1}%
    \setlength\tabcolsep{0pt}%
    \put(0,0){\includegraphics[width=\unitlength,page=1]{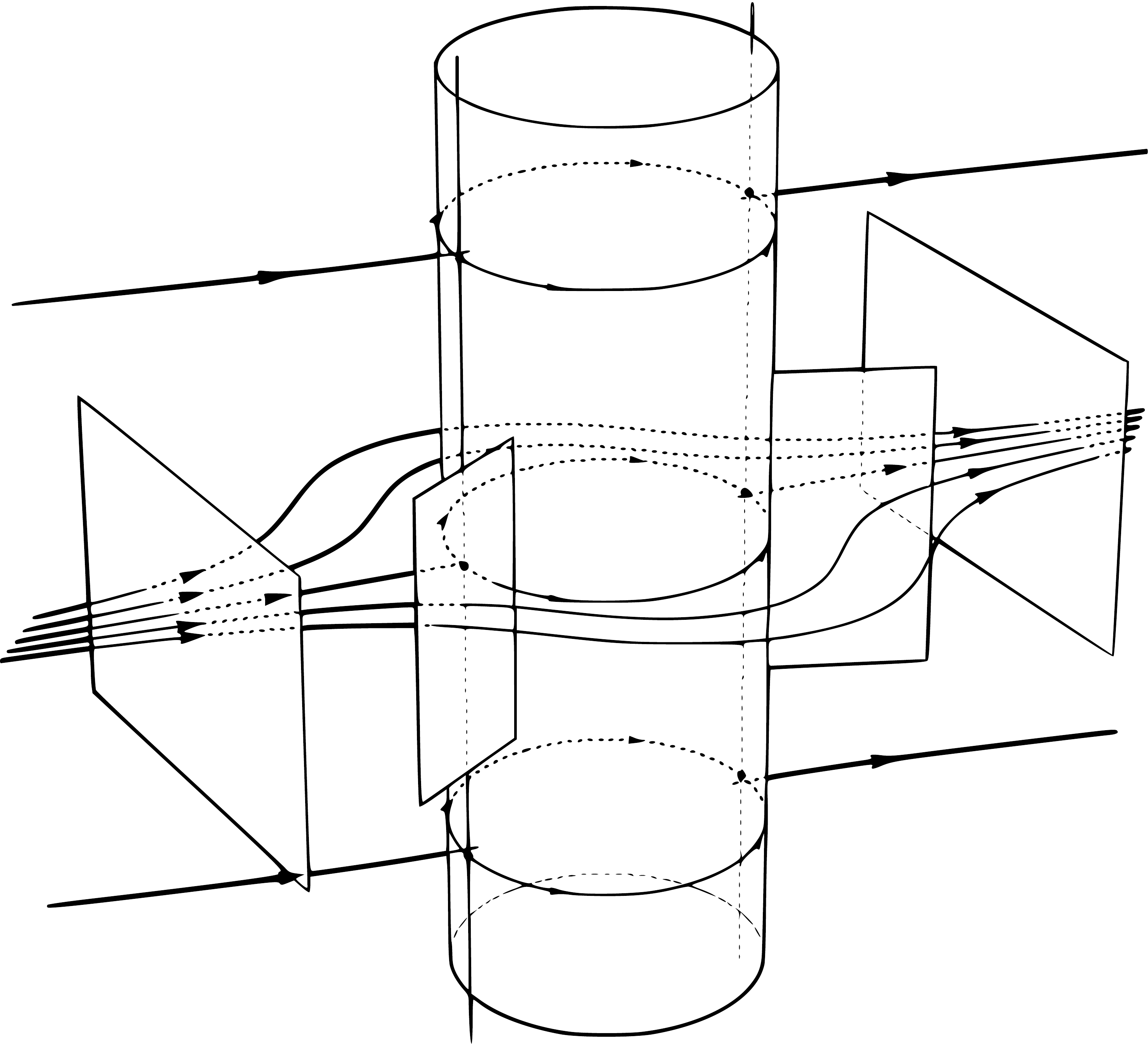}}%
    \put(0.40546216,0.85540555){\color[rgb]{0,0,0}\makebox(0,0)[lt]{\lineheight{2}\smash{\begin{tabular}[t]{l}$ \NHIM^+ $\end{tabular}}}}%
    \put(0.65998705,0.88972766){\color[rgb]{0,0,0}\makebox(0,0)[lt]{\lineheight{2}\smash{\begin{tabular}[t]{l}$ \NHIM^- $\end{tabular}}}}%
    \put(0.20850837,0.21969934){\color[rgb]{0,0,0}\makebox(0,0)[lt]{\lineheight{2}\smash{\begin{tabular}[t]{l}$ \Sigma_0 $\end{tabular}}}}%
    \put(0.90814881,0.41089602){\color[rgb]{0,0,0}\makebox(0,0)[lt]{\lineheight{2}\smash{\begin{tabular}[t]{l}$ \Sigma_3 $\end{tabular}}}}%
    \put(0.76725109,0.55785382){\color[rgb]{0,0,0}\makebox(0,0)[lt]{\lineheight{2}\smash{\begin{tabular}[t]{l}$ \Sigma_1^+ $\end{tabular}}}}%
    \put(0.37091947,0.29726888){\color[rgb]{0,0,0}\makebox(0,0)[lt]{\lineheight{2}\smash{\begin{tabular}[t]{l}$ \Sigma_2^+ $\end{tabular}}}}%
  \end{picture}%
\endgroup%

			\caption{A geometric sketch of the required computation. The dynamics local to the manifold has been split into hyperbolic regions ($ \Sigma_0\to\Sigma_1^+ $ and $ \Sigma_2^+\to\Sigma_3 $) and a smooth transition region ($ \Sigma_1^+\to\Sigma_2^+ $).}
		\end{figure}
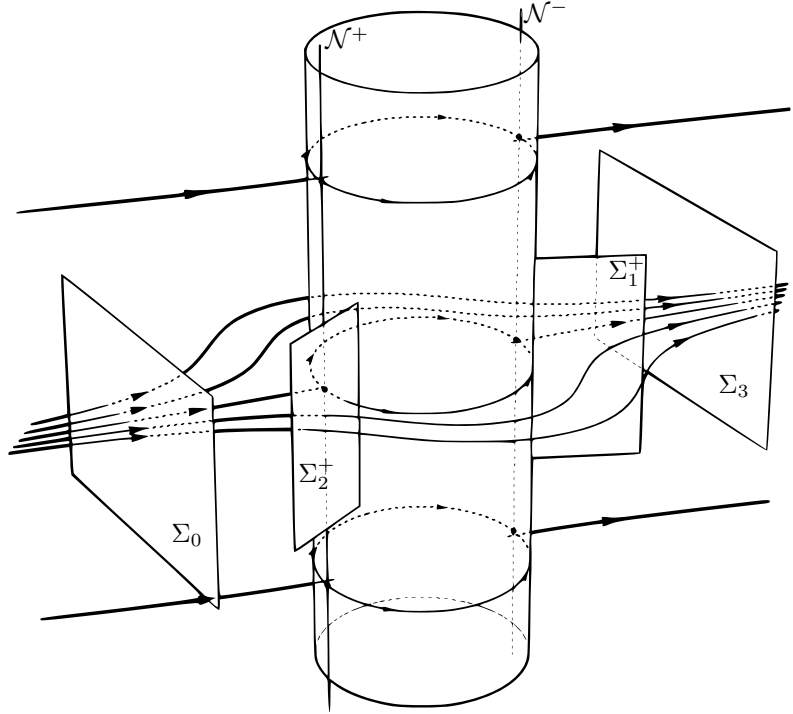
		
		Now, the collision-ejection manifold $ \CE := \CE^+ \cup \CE^- $ splits $ \mathcal{T} $ into two disjoint segments, one containing $ \mathcal{V}^+ $ and the other $ \mathcal{V}^- $. It follows that $ \pib $ can be split into its restriction to these two segments, say $ \pib^+ $ and $ \pib^- $. The key to the calculation is to introduce two intermediate sections, $ \Sigma_1^+ \subset \mathcal{U}^+\cap\mathcal{V}^+ $ and $ \Sigma_2^+ \subset \mathcal{V}^+\cap\mathcal{U}^- $, that are both transversal to the flow and intersect the heteroclinic connection (see Figure \ref{fig:uncoupledsections}). In doing so, the block map $ \pib^+ $ can be decomposed into $ \pib^+ = D_2^+\circ T^+\circ D_1^+ $ where \[ D_1^+:n^+\to\Sigma_1^+,\ T^+:\Sigma_1^+\to\Sigma_2^+ \text{ and } D_2^+:\Sigma_2^+\to n^-. \] Analogously, take $ \pib^- = D_2^-\circ T^-\circ D_1^- $ where $ D_1^-:n^+\to\Sigma_1^-,\ T^-:\Sigma_1^-\to\Sigma_2^- $ and $ D_2^-:\Sigma_2^-\to n^- $. 
		
		$ D_1^+, D_2^+ $ are transitions near a normally hyperbolic manifold of hyperbolic saddles and $ T^+ $ is a regular transition map. It is now obvious what needs to be done; first compute the hyperbolic passages $ D_1^\pm,D_2^\pm $, glue them together with the relevant regular transition map $ T^\pm $ to get $ \pib^\pm $, and compare the one-sided asymptotics towards $ \mathcal{E}^{\pm} $ of $ \pib^+ $ and $ \pib^- $. In fact, with some knowledge of hyperbolic transition maps, it is already possible from this sketch to see how the finite differentiability of the block map will creep in.
		
	\subsection{The Block-map is Quasi-Regular}\label{sec:BlockMapQuasiRegular}
		Due to their relevance to Hilbert's sixteenth problem, the asymptotic properties of hyperbolic transition maps have been well studied for planar vector fields, see for instance \cite{roussarieBifurcationPlanarVector1998,Dumortier1977}. In this context they are called Dulac maps. The Dulac maps have been shown to be \textit{quasi-regular} (also referred to as \textit{almost regular}). We extend the definition of quasi-regular as given by Roussarie \cite{roussarieBifurcationPlanarVector1998} to higher dimensions.
		\begin{definition}
			Let $ (x,u)\in\R^+\times\R^k $. A germ of a map $ f:\R^+\times\R^{k}\to\R^+\times\R^{k} $ at $ 0 $ is called \textit{quasi-regular in $ x $} if there exists $ a,b >0 $ such that
			\begin{enumerate}[i.]
				\item $ f $ has a representative on $ [0,a)\times(-b,b)^{k} $ that is $ C^\infty $ on $ (0,a)\times(-b,b)^{k} $.
				\item $ \lim_{x\to 0} f(x,u) = (0,Au) $, for some linear map $ A:\R^{k}\to\R^{k} $.
				\item The components of $ f - (0, A u )= (f_0,f_1,\dots,f_{k}) $ are asymptotic to the Dulac series,
				\begin{equation*}
					\hat{f}_k(x,u) = \sum_{j=1}^\infty x^{\rho_j} P^k_{j}(u,\ln x),
				\end{equation*}
				with $ 0\neq\rho_j\in \R^+ $, $ P^k_{j} $ is a sequence of polynomials in $ x $ with coefficients smooth in $ u $, and the sum is taken with ordering $ 0 < \rho_j \leq \rho_{j+1}  $. 
				
				Define $ f $ as a \textit{quasi-regular homeomorphism in $ x $} if $ f $ is quasi-regular and $ P^1_1(x,u) = p(u) $ with $ p(u) $ positive on $ (-b,b)^{k} $ and $ A $ is invertible.
			\end{enumerate}
		\end{definition}
		\begin{remark}\label{rmk:QuasiRegularGroup}
			The set of all quasi-regular homeomorphisms in $ x $, denoted by $ \mathcal{D} $, is a group under composition. Further, the group of diffeomorphisms with $ f(0,u)= (0, A u) $, for some invertible $ A $, and $ \frac{\partial f_1}{\partial x}(0,0) > 0 $, is a subgroup of $ \mathcal{D} $. 
		\end{remark}
		
		In the current work we wish to obtain asymptotic properties of transition maps near manifolds of normally hyperbolic saddle singularities. Specifically, let $ \mathcal{N} $ be a manifold of normally hyperbolic saddle singularities of co-dimension 2 inside some vector field $ X $. Take $ u_0 \in \mathcal{N} $ and denote the non-zero eigenvalues of $ X $ at $ u_0 $ by $ \lambda_1(u_0) < 0 < \lambda_2(u_0) $. Without loss of generality take coordinates $ (x,y,u) \in \R\times\R \times \R^k $ local to $ u_0 $ such that $ u_0 $ is at the origin and $ \NHIM $ is given by $ x=y=0 $. Moreover, assume that centre-stable $ W^s(\NHIM) $ and centre-unstable $ W^u(\NHIM) $ manifolds of $ u_0 $ are aligned with the $ x,y $ axis.
		
		From the work of \cite{duignanNormalFormsManifolds}, many details about normal forms and transition maps near manifolds of normally hyperbolic singularities are known. Let $ u_0 \in \NHIM $. A crucial object in the study of the transition map is $ \rho(u_0) := -\lambda_1(u_0)/\lambda_2(u_0) $, the so called \textit{ratio of hyperbolicity}. From Proposition \ref{prop:structureOfCollisionManifold}, the ratio of hyperbolicity of $ \NHIM^+, \NHIM^{-} $ takes the constant value $ 1/3 $ and $ 3 $ respectively for any $ u_0 \in \NHIM^+ $ or $ \NHIM^- $. The following proposition from \cite{duignanNormalFormsManifolds} gives the normal form near a point $ u_0 \in\NHIM $ assumed to be $ 0 $.
		
		\begin{proposition}[\cite{duignanNormalFormsManifolds}]\label{prop:normalformnhim}
			If, for all $ u_0 \in \NHIM $, $ \rho = \frac{p}{q} \in \Q,$ with $ p,q $ co-prime, then there exists a $ C^\infty $, near identity transformation $ \Phi $ and a smooth time rescaling bringing $ X $ into the normal form,
			\begin{equation}\label{eqn:normalformnhim}
				\begin{aligned}
					\dot{x}	&= x \\
					\dot{y}	&= -\rho y + \frac{1}{q} y \sum_{j \geq 1} \alpha_{j}(u) (x^p y^q)^j \\
					\dot{u}_i &= \sum_{j \geq 1} \delta^i_{j}(u)  (x^p y^q)^j, \quad i=1,\dots,k
				\end{aligned}
			\end{equation}
			with $ \alpha_j(u), \delta_j(u) $ smooth functions in $ u $. If $ \rho \nin \Q $ then $ \alpha_j = \delta_j = 0 $.
		\end{proposition}
		
		Now, denote the normal form of $ X $ by $ X_N $ and consider the transverse sections $ \sigma_0 = \{ y = 1 \} $ and $ \sigma_1 = \{ x=1 \} $ in $ X_N $. Let $ (x_0,u_0) ,(y_1,u_1) $ be coordinates on $ \sigma_0,\sigma_1 $ respectively. Define the transition map $ D : \sigma_0 \to \sigma_1 $ which is given in components by
		\begin{equation*}
			y_1 = D_y(x_0,y_0,u_0), \quad  u_1 = D_u(x_0,y_0,u_0).
		\end{equation*} 
		As in the literature on planar vector fields, call this specific transition map the \textit{Dulac map}. A diagram of the Dulac map is given in Figure \ref{fig:1dN} for the case $ \mathcal{N} $ is dimension 1 inside $ \R^3 $.
		\begin{figure}[ht]
			\centering
	\def\svgwidth{0.5\textwidth}
\begingroup%
  \makeatletter%
  \providecommand\color[2][]{%
    \errmessage{(Inkscape) Color is used for the text in Inkscape, but the package 'color.sty' is not loaded}%
    \renewcommand\color[2][]{}%
  }%
  \providecommand\transparent[1]{%
    \errmessage{(Inkscape) Transparency is used (non-zero) for the text in Inkscape, but the package 'transparent.sty' is not loaded}%
    \renewcommand\transparent[1]{}%
  }%
  \providecommand\rotatebox[2]{#2}%
  \newcommand*\fsize{\dimexpr\f@size pt\relax}%
  \newcommand*\lineheight[1]{\fontsize{\fsize}{#1\fsize}\selectfont}%
  \ifx\svgwidth\undefined%
    \setlength{\unitlength}{669bp}%
    \ifx\svgscale\undefined%
      \relax%
    \else%
      \setlength{\unitlength}{\unitlength * \real{\svgscale}}%
    \fi%
  \else%
    \setlength{\unitlength}{\svgwidth}%
  \fi%
  \global\let\svgwidth\undefined%
  \global\let\svgscale\undefined%
  \makeatother%
  \begin{picture}(1,0.71300448)%
    \lineheight{1}%
    \setlength\tabcolsep{0pt}%
    \put(0,0){\includegraphics[width=\unitlength,page=1]{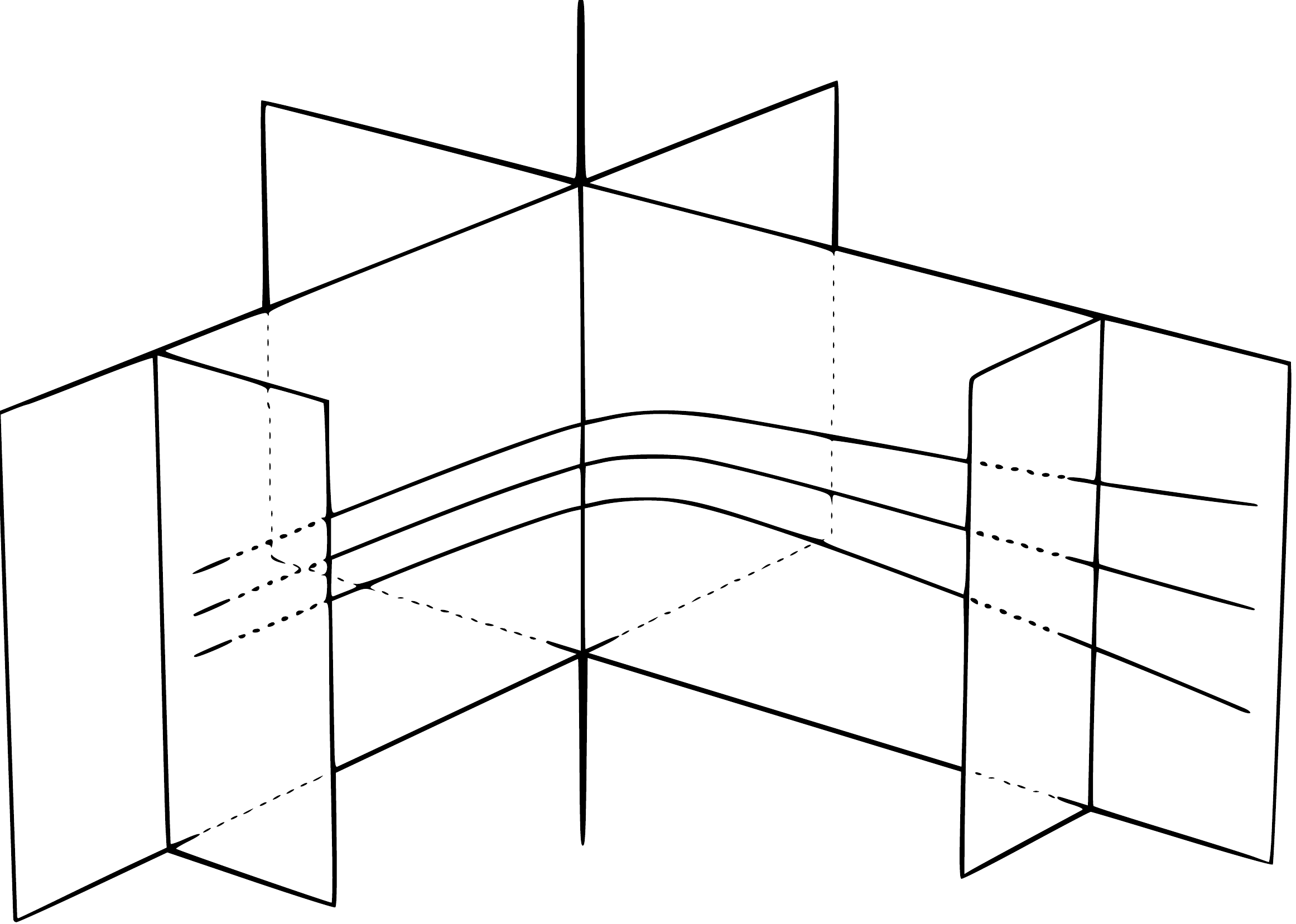}}%
    \put(0.19226834,0.18158677){\color[rgb]{0,0,0}\makebox(0,0)[lt]{\lineheight{2}\smash{\begin{tabular}[t]{l}$\Sigma_0$\end{tabular}}}}%
    \put(0.45610324,0.63021289){\color[rgb]{0,0,0}\makebox(0,0)[lt]{\lineheight{2}\smash{\begin{tabular}[t]{l}$\NHIM$\end{tabular}}}}%
    \put(0.22858569,0.45076244){\color[rgb]{0,0,0}\makebox(0,0)[lt]{\lineheight{2}\smash{\begin{tabular}[t]{l}$W^s(\NHIM)$\end{tabular}}}}%
    \put(0.51271553,0.46464846){\color[rgb]{0,0,0}\makebox(0,0)[lt]{\lineheight{2}\smash{\begin{tabular}[t]{l}$W^u(\NHIM)$\end{tabular}}}}%
    \put(0.76480075,0.18051859){\color[rgb]{0,0,0}\makebox(0,0)[lt]{\lineheight{2}\smash{\begin{tabular}[t]{l}$\Sigma_1$\end{tabular}}}}%
    \put(0,0){\includegraphics[width=\unitlength,page=2]{1dN.pdf}}%
    \put(0.48565879,0.29305695){\color[rgb]{0,0,0}\makebox(0,0)[lt]{\lineheight{2}\smash{\begin{tabular}[t]{l}$D$\end{tabular}}}}%
  \end{picture}%
\endgroup%

			\caption{Diagram of the case $ \mathcal{N} $ is co-dimension 2 in $ \R^3 $}
			\label{fig:1dN}
		\end{figure}
		
		The following proposition from \cite{duignanNormalFormsManifolds} gives the asymptotic structure of the Dulac map $ D $.
		\begin{proposition}[\cite{duignanNormalFormsManifolds}]\label{prop:asymStructureofD}
			If, for all $ u_0\in\NHIM $, $ \rho(u_0)  = \frac{p}{q} \in \Q $ with $ p,q $ co-prime, then the Dulac map $ D = (D_x, D_u) $ has the asymptotic series
			\begin{equation}
				\begin{aligned}
					D_y(x_0,u_0)	&\sim x_0^{\rho} \left(1+\sum_{j \geq 1} P_y^{j}(u_0; \ln x_0) x_0^{ j p} \right)  \\
					D_u(x_0,u_0)	&\sim u_0 + \sum_{j \geq 1} P_u^{j}(u_0; \ln x_0) x_0^{j p} ,
				\end{aligned}
			\end{equation}
			where $ P_y^{j},P_{u_i}^{j} $ are polynomial in $ \ln x_0 $ with coefficients smooth in $ u_0 $ and $ P_y^{j}(0,u_0) = P_{u_i}^{j}(0,u_0) = 0 $. Moreover, $ P_y^{j} $ and $ P_{u_i}^j $ are polynomial in $ \alpha_{l}(u_0),\delta_{l}(u_0) $ for $ l \leq j $ with vanishing constant term. 
			
			If $ \rho(u_0) \nin \Q $ then $ P_y^{j} = P_u^{j} = 0 $ for all $ j \geq 1 $. In either case $ D $ is quasi-regular.
		\end{proposition}
		
		In contrast to the quasi-regularity of the hyperbolic transition maps, the regular transition maps $ T^\pm $ are smooth. This can be deduced from the fact that $ V^\pm $ has no singularities and a consequent use of the flow-box theorem. Combining this fact with Remark \ref{rmk:QuasiRegularGroup}, the following theorem is concluded.
		
		\begin{thm}\label{thm:blockMapIsQuasiRegular}
			The block maps $ \pib^{\pm} $ are quasi-regular. More precisely, if $ Z = (\theta,x,h_1,h_2,y) $ then,
			\[ \bar{\pi}^+(\theta,Z) \sim \gamma_0 Z + \sum_{i,j}^{i+j \rho < m} \gamma_{i,j}(Z) \theta^{i + j \rho} + \gamma_m(Z) 
			\theta^m\ln(\theta) + \dots, \]
			with $ \gamma_{i,j}:\R^{n-1}\to\R^{n-1} $ smooth, $ \rho = 1/3 $ the ratio of hyperbolicity of $ \NHIM^+ $ and $ m \geq 1 $ some integer. A similar asymptotic expansion hold for $ \pib^- $ with perhaps different functions $ \gamma_{i,j},\gamma_m $.
		\end{thm}
		\begin{proof}
			The theorem is proved for the transition map $ \pib^+ $ as $ \pib^- $ follows analogously. Recall that $ \pib^+ = D_1^+ \circ T^+ \circ D_2^+ $ with $ D_1^+ : n^+ \to \Sigma_1^+ $ and $ D_2^+ : \Sigma_2^+ \to n^- $ are hyperbolic transitions and $ T^+ $ is the smooth transition map between $ \Sigma_1^+ $ and $ \Sigma_2^+ $. The exact choice of $ \Sigma_1^+ $ and $ \Sigma_2^+ $ will not affect $ \pib^+ $. In particular they can be chosen so that, using the normal form transformation $ \Phi_+ $ near a point $ (0,0,x,h_1,h_2,y)\in\NHIM^+ $, we can decompose $ D_1 = \phi^{-1}\circ D_+\circ\phi $ where $ D_+ $ is the Dulac map near $ \NHIM^+ $. Likewise, $ D_2 = \Phi_- \circ D_-\circ\Phi_-^{-1} $ with $ D_- $ the Dulac map near $ \NHIM^- $ and $ \Phi_- $ the corresponding normal form transformation. The quasi-regularity of $ \pib^+ $ follows immediately from the fact that $ D_+,D_- $ are quasi-regular, $ \Phi_-,\Phi_+,T^+ $ are diffeomorphism (hence also quasi-regular) and from Remark \ref{rmk:QuasiRegularGroup} that the set of quasi-regular functions forms a group under composition.
			
			To get the specific form of the asymptotic series of $ \pib^+ $ given in the theorem, first note that, from Proposition \ref{prop:structureOfCollisionManifold}, the hyperbolic ratio of $ \NHIM^+ $ is $ \rho = 1/3 $ and of $ \NHIM^- $ is $ 1/\rho = 3 $. Using Proposition \ref{prop:asymStructureofD} and the fact the two ratios of hyperbolicity are multiplicative inverses gives the desired form of $ \pib^+ $.
		\end{proof}
		
		Theorem \ref{thm:blockMapIsQuasiRegular} is precisely what we have been searching for; the finite differentiability of the block map! As $ \rho = 1/3 $, the asymptotic series of the block map is of the form $ \theta^{i+j/3} $ and possibly some $ \theta^m\ln \theta $ terms. Moreover, Theorem \ref{thm:blockMapIsQuasiRegular} shows that the finite differentiability is generic. For the map to be smooth, and hence smoothly regularisable, it is required that all the $ \gamma_{i,j} $ and $ \gamma_m $ vanish. That is, finite differentiability should be expected and any smoothness of the block map is remarkable.
		
		\begin{remark}\label{rmk:orderOfResonance}
			It is important to study what mechanisms give rise to the coefficients $ \gamma_{i,j},\gamma_m $. The $ \theta^m\ln \theta $ terms come from the first resonance term $ \alpha_m(u) (x^p y^q)^{m-1} $ or $ \delta_m^k(u) (x^p y^q)^{m} $ in the normal form \eqref{eqn:normalformnhim}. If $ \alpha_{m} $ (resp. $ \delta_m^k $) does not vanish then a term of type $ \gamma_{mp+1}(u)\theta^{mp+1}\ln\theta $ (resp.  $\gamma_{mp}(u)\theta^{mp}\ln\theta $) arises in the Dulac map. This can be seen from Proposition \ref{prop:asymStructureofD}. The value of $ m $ is called the \textit{order of resonance}. It will be shown that this mechanism does not cause the finite differentiability in the collinear $ 4 $-body problem.
			
			On the other hand, the $ \gamma_{i,j}(u) $ coefficients arise from the interaction of the two Dulac maps $ D_1,D_2 $ and the smooth transition $ T $. Ignoring any higher order resonance terms, then from Proposition \ref{prop:asymStructureofD}, the Dulac maps near $ \mathcal{N} $ behave approximately like 
			\[ D_1(\theta,x,y,h_1,h_2) = (\theta^{1/3}, x, y,h_1,h_2),\quad D_2(r,x,y,h_1,h_2) = (r^3, x, y,h_1,h_2). \]
			Now, if the smooth transition has the form, say, $ T(r,x,y,h_1,h_2) = \left(r,x,y,h_1+ a r^j,h_2+ b r^j\right) $, for some $ a,b\in\R $, then the composition of the maps \[ \pib = D_2\circ T \circ D_1(Z) = \left(\theta,x,y,h_1+ a \theta^{j/3},h_2+ b \theta^{j/3}\right), \]
			showing the arrival of a $ \theta^{j/3} $ term. It is this mechanism that will ultimately lead to the finite differentiability of $ \pib $.
		\end{remark}

	\subsection{Asymptotic Expansion of the Block Map}\label{sec:proofOfMainTheorem}
		With the block map realised as quasi-regular and the mechanisms leading to finite differentiability discussed, we are now in a position to determine the precise regularity of the simultaneous binary collision singularities $ \sbc $ for the collinear 4-body problem. We make a particular choice of sections $ \Sigma_0, \Sigma_3 $ transverse to $ \CE^+,\CE^- $ respectively and the intermediate sections $ \Sigma_i^+ $ to explicitly compute the hyperbolic transitions and smooth transition map. Recall the procedure from Section \ref{sec:GeometricSketch}:
		\begin{enumerate}
			\item Take the truncated normal form $ X^9 $ and blow-up and desingularise the set of simultaneous binary collisions to produce $ X_\theta $. From Proposition \ref{prop:structureOfCollisionManifold}, a collision manifold that is two normally hyperbolic invariant manifolds connected by heteroclinics is obtained.
			\item In the blow-up system $ X_\theta $, compute the hyperbolic transition $ D_1^+ $ and $ D_2^+ $ near an arbitrary point of each normally hyperbolic invariant manifold $ \NHIM^+,\NHIM^- $. This is done by first computing the normal form $ X_N $ to a desired order near each $ \NHIM^+, \NHIM^- $ and using Proposition \ref{prop:asymStructureofD}.
			\item In the blow-up system $ X_\theta $, solve the variational equations to a desired order along $ (r,x,h_1,h_2,y) =0 $ to get the smooth transition map from $ \Sigma_1^+ $ to $ \Sigma_2^+ $.
			\item Compose the maps, $ \pib^+ = D_2^+\circ T^+\circ D_1^+ $ to obtain the asymptotic series of the block map $ \pib^+ $. 
		\end{enumerate}
	
		Whilst theoretically this procedure is sound, there are three obstacles faced in carrying out an explicit calculation. The most easily resolved obstacle is with the blow-up method described in Section \ref{sec:CollisionManifold}. Here, polar coordinates in $ (z_1,z_2) $ were introduced to blow-up the set of simultaneous binary collisions. Whilst this certainly achieved the blow-up, it introduced trig functions into the equations. This is at odds with the normal form methods needed to compute the hyperbolic transitions $ D_i^+ $. The normal form procedure is iterative on the homogeneous components of the vector field $ X_\theta $. In order to get the homogeneous components, one needs to Taylor expand the vector field $ X_\theta $; a computation that involves the Taylor expansion of many trig functions. Further, the expansion needs to be done twice at both $ \NHIM^+ $ and $ \NHIM^- $. 
		
		Fortunately, one can circumvent both the introduction of trig functions and the need to do the normal form procedure twice by performing a \textit{directional blow-up} instead of the \textit{polar blow-up} performed in Section \ref{sec:CollisionManifold}. Firstly, it is more convenient to work with $ X^9 $ rotated clockwise by $ \pi/4 $ in the $ (z_1,z_2) $ plane through the transformation $ (\tilde{z}_1,\tilde{z}_2) = \frac{1}{2}(z_1 + z_2, z_1 - z_2) $. This aligns $ \NHIM^+,\NHIM^- $ with $\theta =  \pi,0 $ in the polar blow-up. Applying the rotation yields
		\begin{equation}
			\begin{aligned}[c]
				\tilde{z}_1^\prime	&=  \tilde{z}_1^2 + \tilde{z}_2^2 + \tilde{R}_{6,1}(\tilde{z}_1,\tilde{z}_2)	\\
				\tilde{z}_2^\prime	&=  -2 \tilde{z}_1 \tilde{z}_2 + \tilde{R}_{6,2}(\tilde{z}_1,\tilde{z}_2) 	\\
				x^\prime	&= 0 	
			\end{aligned}
			\qquad
			\begin{aligned}[c]
				h_1^\prime	&= b_c a_1^{-1/3} \tilde{R}_h(\tilde{z}_1,\tilde{z}_2) 	\\
				h_2^\prime	&= -b_c a_2^{-1/3} \tilde{R}_h(\tilde{z}_1,\tilde{z}_2) 	\\
				y^\prime	&= 0
			\end{aligned}
		\end{equation}
		where 
		\begin{equation}
			\begin{aligned}
				\tilde{R}_{6,1}(\tilde{z}_1,\tilde{z}_2) &= \frac{1}{2}\left( R_{6,1}(\tilde{z}_1+\tilde{z}_2,\tilde{z}_1-\tilde{z}_2) + R_{6,2}(\tilde{z}_1+\tilde{z}_2,\tilde{z}_1-\tilde{z}_2) \right) \\
				\tilde{R}_{6,2}(\tilde{z}_1,\tilde{z}_2) &= \frac{1}{2}\left( R_{6,1}(\tilde{z}_1+\tilde{z}_2,\tilde{z}_1-\tilde{z}_2) - R_{6,2}(\tilde{z}_1+\tilde{z}_2,\tilde{z}_1-\tilde{z}_2) \right) \\
				\tilde{R}_h(\tilde{z}_1,\tilde{z}_2) &= R_h(\tilde{z}_1+\tilde{z}_2,\tilde{z}_1-\tilde{z}_2)
			\end{aligned}
		\end{equation}
		
		The directional blow-up is no more difficult, it merely involves taking charts of $ S^1 $ by performing the transformations,
		\begin{equation}\label{eqn:directionalBlowUp}
			\begin{aligned}
				\text{$ \tilde{z}_1 $-direction, } (z_1,z_2) &= (\hat{u},\hat{u}\hat{v}) \\
				\text{$ \tilde{z}_2 $-direction, } (z_1,z_2) &= (\bar{u}\bar{v},\bar{v}),
			\end{aligned}
		\end{equation}
		followed by rescaling $ d\hat{\tau} = \hat{u}d\tau, d\bar{\tau} = \bar{v}d\tau $ to produce the desingularised $ z_1 $ and $ z_2 $ directional blow-ups $ \hat{X},\bar{X} $. 

		The $ \tilde{z}_1 $ and $ \tilde{z}_2 $-directional blow-up, and the order to which we know them after truncating the normal form $ X^9 $ at degree 9, are,
		\begin{equation}\label{eqn:xblowup}
			\begin{aligned}[c]
				\hat{u}^\prime	&=  \hat{u}\left(1+\hat{v}^2\right) + \hat{u}^5\tilde{R}_{6,1}(1,\hat{v}) + O(\hat{u}^9) \\
				\hat{v}^\prime	&=  -3 \hat{v} \left(1+1/3 \hat{v}^2\right) + \hat{u}^4\left(\tilde{R}_{6,1}(1,\hat{v}) - \hat{v}\tilde{R}_{6,2}(1,\hat{v})\right) + O(\hat{u}^9) \\
				x^\prime		&= 0 + O(\hat{u}^9)	\\
			\end{aligned}
			\qquad
			\begin{aligned}	
				h_1^\prime	&= b_c a_1^{-1/3} \hat{u}^8\tilde{R}_h(1,\hat{v})  + O(\hat{u}^9)	\\
				h_2^\prime	&= -b_c a_2^{-1/3}\hat{u}^8\tilde{R}_h(1,\hat{v}) + O(\hat{u}^9)	\\
				y^\prime	&= 0 + O(\hat{u}^9)
			\end{aligned}
		\end{equation}
		and
		\begin{equation}\label{eqn:yblowup}
			\begin{aligned}[c]
				\bar{u}^\prime	&=  1+3\bar{u}^2 + \bar{v}^4\left(\tilde{R}_{6,2}(\bar{u},1) - \bar{u}\tilde{R}_{6,1}(\bar{u},1)\right) + O(\bar{v}^9)	\\
				\bar{v}^\prime	&=  -2 \bar{u}\bar{v} + \bar{v}^5\tilde{R}_{6,2}(\bar{u},1) + O(\bar{v}^9)  \\
				x^\prime		&= 0 + O(\bar{v}^9)	\\	
			\end{aligned}
			\qquad
			\begin{aligned}
				h_1^\prime	&= b_c a_1^{-1/3} \bar{v}^8\tilde{R}_h(\bar{u},1) + O(\bar{v}^9)	\\
				h_2^\prime	&= -b_c a_2^{-1/3} \bar{v}^8\tilde{R}_h(\bar{u},1) + O(\bar{v}^9) 	\\
				y^\prime	&= 0 + O(\bar{v}^9)
			\end{aligned}
		\end{equation}
		respectively. 
		
		As desired, system \eqref{eqn:xblowup} and \eqref{eqn:yblowup} are free of trig functions and polynomial in the variables. The collision manifold $ r = 0 $ from the polar blow-up has become the manifolds $ \hat{u}=0 $ and $ \bar{v}=0 $ in the two charts. Further, the two normally hyperbolic invariant manifolds in the polar blow-up $ X_\theta $ have been reduced to a single normally hyperbolic manifold in the $ \tilde{z}_1 $-chart at $ (\hat{u},\hat{v}) = (0,0) $ which we denote by $ \NHIM $. In the $ \tilde{z}_2 $-chart $ \bar{X} $, the collision manifold $ \bar{v} = 0 $ is free of singularities and the flow is given trivially on it. A projection of the charts into the $ (\tilde{z}_1,\tilde{z}_2) $, $ (\hat{u},\hat{v}) $ and $ (\bar{u},\bar{v}) $ planes is provided in Figure \ref{fig:DirectionalBlowUp}.
		
		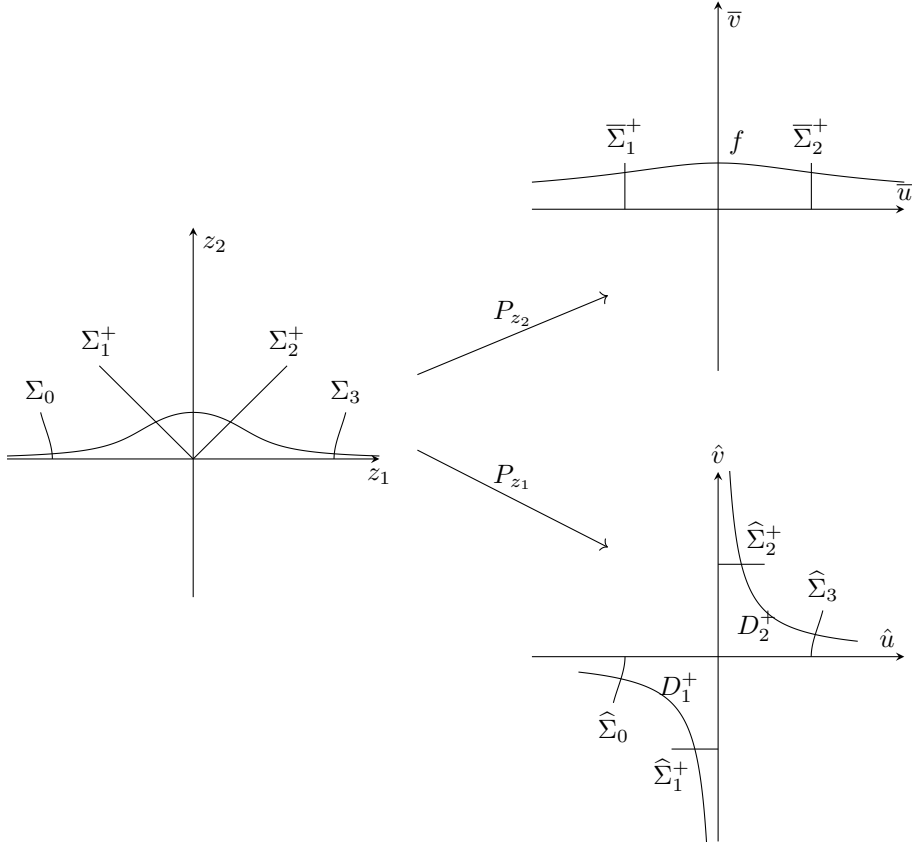
\begin{figure}[ht]		
			\centering
			\begin{tikzpicture}
				\begin{axis}[name = plot1,
					scale only axis,
					width = 1/3*\textwidth,
					height =1/3*\textwidth ,
					axis x line=middle,
					axis y line=middle,
					x label style = {anchor=north},
					axis equal,
					xlabel = {$z_1$},
					ylabel = {$z_2$},
					restrict y to domain = -4:2,
					restrict x to domain = -2:2]
					\addplot [domain = 0:0.5, samples = 300]
					({sqrt(0.125/x -x^2)}, {x});
					\addplot [domain = 0:0.5, samples = 300]
					({-sqrt(0.125/x -x^2)}, {x}); 
					\addplot [domain = 0:0.5, samples = 300]
					(1.5+x^2-x^3, {x}) node[above,pos=1]{$ \Sigma_3 $};
					\addplot [domain = 0:0.5, samples = 300]
					(-1.5-x^2+x^3, {x}) node[above,pos=1]{$ \Sigma_0 $};
					\addplot [domain = 0:1, samples = 300]
					({x}, {x}) node[above,pos=1]{$ \Sigma_2^+ $};
					\addplot [domain = 0:-1, samples = 300]
					({x}, {-x})node[above,pos=1]{$ \Sigma_1^+ $};
				\end{axis}
				
				\begin{axis}[name = plot2,at={($(plot1.east)+(2cm,3cm)$)},anchor=outer west,axis x line=middle,
					scale only axis,
					width = 1/3*\textwidth,
					height =1/3*\textwidth ,
					axis y line=middle,
					x label style = {anchor = south},
					axis equal,
					xlabel = {$\overline{u}$},
					ylabel = {$\overline{v}$},
					restrict y to domain = -2:2,
					restrict x to domain = -2:2]
					\addplot [domain = 0:0.5, samples = 300]
					({sqrt(0.125/x^3 -1)}, {x}) node[above,pos=0.9]{$ f $};
					\addplot [domain = 0:0.5, samples = 300]
					({-sqrt(0.125/x^3 -1)}, {x});
					\addplot [domain = 0:0.5, samples = 300]
					(1, {x}) node[above,pos=1]{$ \overline{\Sigma}_2^+ $};
					\addplot [domain = 0:0.5, samples = 300]
					(-1, {x}) node[above,pos=1]{$ \overline{\Sigma}_1^+ $};  
				\end{axis}
				
				\begin{axis}[name = plot3,at={($(plot1.east)+(2cm,-3cm)$)},anchor=outer west,axis x line=middle,
					scale only axis,
					width = 1/3*\textwidth,
					height =1/3*\textwidth ,
					axis y line=middle,
					y label style = {anchor=south},
					axis equal,
					xlabel = {$\hat{u}$},
					ylabel = {$\hat{v}$},
					ymin = -2,ymax=2,
					xmin = -2, xmax = 2]
					\addplot [domain = 0.09:1.5, samples = 300]
					({x}, {0.25/x}) node[below,pos=0.65]{$ D_2^+ $}; 
					\addplot [domain = -0.09:-1.5, samples = 300]
					({x}, {0.25/x}) node[above,pos=0.65]{$ D_1^+ $}; 
					\addplot [domain = 0:0.5, samples = 300]
					(1+x^2-x^3, {x}) node[above,pos=1]{$ \widehat{\Sigma}_3 $};
					\addplot [domain = -0.5:0, samples = 300]
					(-1-x^2-x^3, {x}) node[below,pos=0]{$ \widehat{\Sigma}_0 $};
					\addplot [domain = 0:0.5, samples = 300]
					(x, 1) node[above,pos=1]{$ \widehat{\Sigma}_2^+ $};
					\addplot [domain = -0.5:0, samples = 300]
					(x, -1) node[below,pos=0]{$ \widehat{\Sigma}_1^+ $};
				\end{axis}
				
				\draw[->] ($(plot1.east) + (0.5cm,0.5cm)$) -- ($ (plot2.south west) +(1cm,1cm)  $) node[midway,above]{$ P_{z_2} $};
				\draw[->] ($(plot1.east) + (0.5cm,-0.5cm)$) -- ($ (plot3.north west) + (1cm,-1cm)$) node[midway,above]{$ P_{z_1} $};

			\end{tikzpicture}
			\caption{Intermediate sections and their desingularisations for the upper block map $ \pib_+ $.}\label{fig:DirectionalBlowUp}
		\end{figure}
		
		The second obstacle again involves the normal form procedure near the normally hyperbolic manifold. In the computation, it is desired to choose the nicest possible sections $ \Sigma_0, \Sigma_3 $ transverse to $ \CE^+,\CE^- $ respectively. Due to the complicated form of the asymptotic series given in Theorem \ref{thm:blockMapIsQuasiRegular}, it is wise to choose sections that minimise the complexity of $ D_1^+ $ and $ D_2^+ $. To do this, let $ \Phi $ be the normalising transform near the normally hyperbolic invariant manifold $ \NHIM $ and $ P_{\tilde{z}_1},P_{\tilde{z}_2} $ the maps from $ X^9 $ to $ \hat{X},\bar{X} $ respectively. Then choose the sections
		\begin{equation}
			\begin{aligned}
				\Sigma_0 &= P_{\tilde{z}_1}^{-1}\circ\Phi(\sigma_0),\quad \sigma_0 := \{-1\}\times(-\delta,\delta)\times B_\delta(0)^4 \\
				\Sigma_3 &= P_{\tilde{z}_1}^{-1}\circ\Phi(\sigma_3),\quad \sigma_3 := \{1\}\times(-\delta,\delta)\times B_\delta(0)^4 \\
				\Sigma_1^+ &= P_{\tilde{z}_1}^{-1}\circ\Phi(\sigma_1),\quad \sigma_1 := (-\delta,0]\times\{-1\}\times B_\delta(0)^4 \\
				\Sigma_2^+ &= P_{\tilde{z}_1}^{-1}\circ\Phi(\sigma_2),\quad \sigma_2 := [0,\delta)\times\{1\}\times B_\delta(0)^4 \\
			\end{aligned}
		\end{equation}
		with some choice of $ 0 < \delta \leq 1 $ and $ B^4_\delta(0) $ the open ball of radius $ \delta $ in $ \R^4 $.
		Further, denote respectively by $ \hat{\Sigma}_i,\bar{\Sigma}_i $ the images of $ \sigma_i $ in the $ \tilde{z}_1 $ and $ \tilde{z}_2 $ directional charts. See Figure \ref{fig:DirectionalBlowUp} for a depiction of the sections. With these sections, coordinates on all images of each $ \sigma_i $ can be given by the normal form coordinates, between which the hyperbolic transition maps are simply the Dulac map.
		
		The third obstacle still arises when trying to compute the transition $ T^+ $ between the normally hyperbolic sectors. The transition is between $ \sigma_1 $ and $ \sigma_2 $ in the normal form $ X_N $. In order to calculate the transition $ T $, coordinates in $ X_N $ need to be transformed to coordinates in $ \bar{X} $. Here lies the problem; at any iteration in its computation, $ \Phi $ is only known up to some truncated order. Consequently, the image of $ \sigma_i $ in $ \bar{X} $ will only be known to some truncated order.  The key to clearing this obstacle is to observe that the block map $ \pib^+ $ is independent of the choice of intermediate sections $ \Sigma_1^+ $ and $ \Sigma_2^+ $. Hence, there is freedom in the choice of these sections. In particular, sections in $ X_N $ can be chosen by
		\begin{equation}
			\begin{aligned}
				\sigma_1^\nu &:= (-\delta,0]\times\{-\nu\}\times B_\delta(0)^4 \\
				\sigma_2^\nu &:= [0,\delta)\times\{\nu\}\times B_\delta(0)^4 \\
			\end{aligned}
		\end{equation}
		for $ 0 < \nu \ll 1 $ and $ \bar{\Sigma}_i^\nu,\hat{\Sigma}_i^\nu,\Sigma_i^\nu $ the image of $ \sigma_i $ in $ \bar{X},\hat{X},X $ respectively. As the final series does not depend on this choice, then it must be that the series does not depend on $ \nu $, and inevitably, the limit $ \nu \to 0 $ can be taken. Essentially, $ \nu $ is a book keeping measure to ensure the images of $ \sigma_i $ are known to sufficiently high order.

		With the obstacles adequately navigated, we are ready to proceed with the calculation of $ \pib^+ $. 
		\subsubsection{Computing the hyperbolic transitions}
			Take $ (0,0,x^*,h_{1}^*,h_{2}^*,y^*) \in \NHIM $ in the $ \tilde{z}_1 $-directional blow-up $ \hat{X} $. We wish to compute the hyperbolic transition maps $ D_1^\nu, D_2^\nu $ in a neighbourhood of this point. From Proposition \ref{prop:asymStructureofD} and Proposition \ref{prop:normalformnhim} this can be done by first computing the normal form $ X_N $ near $ \NHIM $.
			
			It is possible to iteratively compute the normal form of $ \hat{X} $ through the cohomological equations order by order. However, this can be avoided by using the approximate integral $ \kappa $ computed in \eqref{eqn:kappaOriginal}. First, writing $ \kappa $ in the rotated $ \tilde{z}_1,\tilde{z}_2 $ coordinates,
			\begin{equation}\label{eqn:tildekappa}
				\begin{aligned}
					\tilde{\kappa}(\tilde{z}_1,\tilde{z}_2) &= \frac{1}{3} \tilde{z}_2(3\tilde{z}_1^2 + \tilde{z}_2^3) + (h_1+h_1^*)^2 \tilde{\kappa}_7(\tilde{z}_1,\tilde{z}_2) - (h_2+h_2^*)^2 \tilde{\kappa}_7(\tilde{z}_1,-\tilde{z}_2), \\
					\tilde{\kappa}_7(\tilde{z}_1,\tilde{z}_2)	&= \kappa_7(\tilde{z}_1+\tilde{z}_2,\tilde{z}_1-\tilde{z}_2).
				\end{aligned}
			\end{equation}
			Projecting $ \tilde{\kappa} $ into the $ \tilde{z}_1 $-directional chart induces an integral to order $ O(\hat{u}^8) $ in $ \hat{X} $ near the collision manifold $ \hat{u} = 0 $,
			\begin{equation}
				\tilde{\kappa} = \hat{u}^3 \left(\hat{v} + \frac{1}{3}\hat{v}^3\right) + \hat{u}^7 \left( (h_1+h_1^*)^2 \tilde{\kappa}_7(1,\hat{v}) - (h_2+h_2^*)^2 \tilde{\kappa}_7(1,-\hat{v}) \right) + O(\hat{u}^8) .
			\end{equation}
			
			As $ \hat{u} = 0 $ is invariant it is possible to write  $ \hat{u}^\prime = \hat{u} G(\hat{u},\hat{v}) $ for some smooth function $ G $. Looking at the form of $ \hat{u}^\prime $ in \eqref{eqn:xblowup} we have that,
			\[ G(\hat{u},\hat{v}) = \left(1+\hat{v}^2\right) + \hat{u}^4\tilde{R}_{6,1}(1,\hat{v}) + O(\hat{u}^8). \]
			The normal form to sufficiently higher order can then be computed by making the smooth time rescaling 
			\[ d\tau = G(\hat{u},\hat{v})^{-1} d\tilde{\tau}, \]
			and introducing $ u,v,h_i $ as the normal form coordinates through the near identity transformation $ \Phi $,
			\begin{equation}\label{eqn:normalFormCoords}
				\Phi:\quad u = \hat{u},\quad v = \hat{u}^{-3} \tilde{\kappa}(\hat{u},\hat{v},h_1,h_2),\quad h_1 = h_1 - u^8 b_c \frac{216 \hat{v} }{95 a_1^{1/3} },\quad h_2 = h_2 + u^8 b_c \frac{216 \hat{v} }{95 a_2^{1/3} }.
			\end{equation}
			
			Noting that $ \frac{d}{d\tilde{\tau}}\tilde{\kappa}= 0 + O(u^8) $, the normal form $ X_N $ near $ \NHIM $ is given by,
			\begin{equation}
				\begin{aligned}
					u^\prime &= u + O(u^9) \\
					v^\prime &= -3 v + O(u^8) \\
					x^\prime&= 0 + O(u^9)
				\end{aligned}
				\quad
				\begin{aligned}
				h_1^\prime &= h_1 + O(u^8,v^2) \\
				h_2^\prime &= h_2 + O(u^8,v^2) \\
				y^\prime&= 0 + O(u^9),
				\end{aligned}
			\end{equation}
			with $ {}^\prime $ denoting derivative with respect to $ \tilde{\tau} $.

			The truncated normal form $ X_N $, is remarkably simple; it is merely the leading order terms of $ \hat{X} $. This truncation admits $ x,y,h_1,h_2 $ and  $ \tilde{\kappa} = u^3 v $ as integrals. From these integrals the hyperbolic transitions  \( D_1^\nu:(-1,-v,x,h_{1},h_{2},y) \mapsto (-u,-\nu,x,h_{1},h_{2},y) \text{ and } D_2^\nu:(u,\nu,x,h_{1},h_{2},y)\mapsto(1,v,x,h_1,h_2,y) \) are easily computed. What needs to be determined is the order to which $ D_1 $ and $ D_2 $ is known if the normal form is truncated at order $ 9 $. From Remark \ref{rmk:orderOfResonance} any resonant monomial with non-vanishing coefficient appearing in the normal form will produce terms of type $ u^{mp}\ln(u) $, where $ p=1 $ for $ D_1^\nu $ and $ p=3 $ for $ D_2^\nu $ (because the ratio of hyperbolicity is $ 1/3 $ and $ 3 $ respectively). Now, there are no resonance terms appearing to order $ 8 $ in $ X_N $ and hence we can conclude there are no terms in $ D_1^\nu $ of the from $ v \ln v , v^2 \ln v $, and no terms of the form $ u^3 \ln u, u^6\ln u $ in $ D_2^\nu $.
			
			It follows that the hyperbolic transitions are simply
			\begin{equation}
				\begin{aligned}
						D_1^\nu(v,x,h_1,h_2,y) &= (\nu^{-1/3}v^{1/3},x,h_1,h_2,y) + O(v^{3}\ln v ),\\ D_2^\nu(u,x,h_1,h_2,y) &= (\nu u^3 ,x,h_1,h_2,y) + O(u^9 \ln u).
				\end{aligned}
			\end{equation} 
			
		\subsubsection{Smooth Transition Map}
			We will use the $ \tilde{z}_2 $-directional blow-up, system \eqref{eqn:yblowup}, to compute the smooth transition $ T^+:\sigma_1^\nu\to\sigma_2^\nu $. Recall that,
			\[ \sigma_1^\nu := \{ v = -\nu \},\qquad \sigma_2^\nu := \{ v = \nu \}, \]
			where $ v $ is the normal form coordinate of $ X_N $. The idea is to compute $ T^+ $ by considering $ T^+:\bar{\Sigma}_1^\nu \to \bar{\Sigma}_2^\nu $ with $ \bar{\Sigma}_i^\nu $ the images of $ P_{\tilde{z}_2}\circ P_{\tilde{z}_1}^{-1} \circ \Phi^{-1}(\sigma_i) $ parameterised by the normal form coordinates. The transition $ T^+ $ will be computed using the variational equations. One can compute and solve the variational equations using the coordinates $ \bar{u},\bar{v} $. However, by making use of the approximate integral $ \tilde{\kappa} $, the task becomes much more manageable.
			
			Writing $ \tilde{\kappa} $ in the $ \tilde{z}_2 $-directional chart yields,
			\begin{equation}
				\tilde{\kappa} = (\frac{1}{3} + \bar{u}^2) \bar{v}^3 + \bar{v}^7\left((h_1+h_1^*)^2 \tilde{\kappa}_7(\bar{u},1) - (h_2 + h_2^*)^2 \tilde{\kappa}_7( \bar{u}, -1) \right) + O(\bar{v}^8),
			\end{equation}
			with $ \frac{d}{d\tau}\tilde{\kappa} = 0 + O(\bar{v}^7) $. Hence, by replacing $ \bar{v} $ by the new coordinate $ w $ through the diffeomorphism 
			\[ w = \tilde{\kappa}(\bar{u},\bar{v},h_1,h_2)^{\frac{1}{3}}, \]
			and noting that $ \bar{u}^\prime > 0 $ for $ \bar{v} $ sufficiently small, system \eqref{eqn:yblowup} is transformed to the non-autonomous system,
			\begin{equation}\label{eqn:nonAutonomousdbu}
				\begin{aligned}[c]
					\frac{d w}{d \bar{u}}	&= 0 + O(w^7)  \\
					\frac{d x}{d \bar{u}}	&= 0 + O(w^9)	\\	
				\end{aligned}
				\qquad
				\begin{aligned}
					\frac{d h_1}{d \bar{u}}	&= 3^{8/3} b_c a_1^{-1/3} \left( 1 + 3\bar{u}^2 \right)^{-11/3} \tilde{R}_h(\bar{u},1) w^8 + O(w^9)	\\
					\frac{d h_2}{d \bar{u}}	&= -3^{8/3}  b_c a_2^{-1/3} \left( 1 + 3 \bar{u}^2 \right)^{-11/3} \tilde{R}_h(\bar{u},1) w^8 + O(w^9) 	\\
					\frac{d y}{d \bar{u}}	&= 0 + O(w^9)
				\end{aligned}
			\end{equation} 
			
			System \eqref{eqn:nonAutonomousdbu} has an explicit solutions on the collision manifold given by 
			\[ (w,x,h_1,h_2,y) = (0,\bar{x}_0,\bar{h}_{10},\bar{h}_{20},\bar{y}_0), \]
			for each choice of $ \pmb{x}_0 = (\bar{x}_0,\bar{h}_{10},\bar{h}_{20},\bar{y}_0) \in \R^4 $. We seek a variation of this solutions in the $ w $ direction, that is, we want to compute the variation,
			\begin{equation}\label{eqn:varEquations}
				\begin{aligned}[c]
					w(w_0,\pmb{x}_0,\bar{u})		&= w^{(1)}(\pmb{x}_0;\bar{u})w_0 + \sum w^{(j)}(\pmb{x}_0;\bar{u})w_0^j \\
					x(w_0,\pmb{x}_0,\bar{u})		&= \bar{x}_0 + \sum x^{(j)}(\pmb{x}_0; \bar{u})w_0^j 	\\	
					h_1(w_0,\pmb{x}_0,\bar{u})		&= \bar{h}_{10} + \sum h_1^{(j)}(\pmb{x}_0; \bar{u})w_0^j 	\\
					h_2(w_0,\pmb{x}_0,\bar{u})		&= \bar{h}_{20} + \sum h_2^{(j)}(\pmb{x}_0; \bar{u})w_0^j  	\\
					y(w_0,\pmb{x}_0,\bar{u})		&= \bar{y}_0 + \sum W^{(j)}(\pmb{x}_0; \bar{u})w_0^j 
				\end{aligned}
			\end{equation}
			with $ w^{(1)}(\pmb{x}_0, 0 ) = 1 $ and otherwise $ \eta^{(j)}(\pmb{x}_0, 0 ) = 0, \eta = w,x,h_1,h_2,y $, so that at $ \bar{u} = 0 $ the variation has the initial conditions $ (w,x,h_1,h_2,y) = (w_0,\bar{x}_0,\bar{h}_{10},\bar{h}_{20},\bar{y}_0) $. The coefficient functions $ \eta^{(j)} $ can be computed using the variational equations. These equations are derived by differentiating both sides of \eqref{eqn:varEquations} by $ d/d\bar{u} $, replacing the lhs by the non autonomous system \eqref{eqn:nonAutonomousdbu} and substituting the variables $ (w,x,h_1,h_2,y) $ with their variations. The coefficients of $ w_0^j $ are then equated to get a linear, non-autonomous system in $ \eta^{(j)} $ called the $ j^{th} $ order variational equations.
			
			Due to the absence of lower order $ w $ terms in \eqref{eqn:nonAutonomousdbu}, it is immediate that,
			\begin{align*}
				w^{(1)}(\pmb{x}_0;\bar{u})		&= 1 ,\qquad w^{(j)}(\pmb{x}_0;\bar{u}) = 0,\quad  j=2,\dots,6, \\
				x^{(j)}(\pmb{x}_0;\bar{u}) 		&= 	y^{(j)}(\pmb{x}_0;\bar{u}) = 0,\quad j=1,\dots,8 \\
				h_1^{(j)}(\pmb{x}_0;\bar{u})	&=h_2^{(j)}(\pmb{x}_0;\bar{u}) = 0,\quad j=1,\dots,7 
			\end{align*}
			Moreover, the $ 8^{th} $ variation of $ h_i $ is given by 
			\[ h_1^{(8)} = b_c a_1^{-1/3} \bar{H}^{(8)}(\bar{u}),\qquad h_2^{(8)} = -b_c a_2^{-1/3} \bar{H}^{(8)}(\bar{u}),  \]
			where,
			\begin{align*}
				\bar{H}^{(8)}(\bar{u})	&= 3^{8/3} \int_{0}^{\bar{u}} \left( 1 + 3 u^2 \right)^{-11/3} \tilde{R}_h(u,1) du \\
					&= -\frac{72}{95} 3^{2/3} \bar{u} \left(\frac{9 \left(\bar{u}^4+2 \bar{u}^2-3\right)}{\left(3 \bar{u}^2+1\right)^{5/3}}-38 \, _2F_1\left(\frac{1}{2},\frac{2}{3};\frac{3}{2};-3
					\bar{u}^2\right)\right),
			\end{align*}
			and $ \, _2F_1  $ is the hypergeometric function.
			
			In summary, the variation is computed as,
			\begin{equation}\label{eqn:varsol}
				\begin{aligned}
					w 	&=  w_0 + O(w_0^7) \\
					x 	&= \bar{x}_0 + O(w_0^9) \\
					h_1	&= \bar{h}_{10} + b_c a_1^{-1/3} \bar{H}^{(8)}(\bar{u}) w_0^8 + O(w_0^9)\\
					h_2	&= \bar{h}_{20} - b_c a_2^{-1/3} \bar{H}^{(8)}(\bar{u}) w_0^8 + O(w_0^9) \\
					y	&= \bar{y}_0 + O(w_0^9).				
				\end{aligned}
			\end{equation}
			
			One can think of the variation \eqref{eqn:varsol} as the flow $ \psi_{\bar{u}}(w_0,\pmb{x}_0) $ of the non-autonomous system \eqref{eqn:nonAutonomousdbu} up to some order in $ w_0 $. With this view, a method for computing the smooth transition $ T^+:\bar{\Sigma}_1^\nu \to \bar{\Sigma}_2^\nu $ becomes apparent. Let $ \pmb{w}_i = (w_i,\bar{x}_i,\bar{h}_{1i},\bar{h}_{2i},\bar{y}_i) $ be the coordinates on $ \bar{\Sigma}_i^\nu $. Then there exists $ \bar{u}_i = \bar{u}_i(\pmb{w}_i) $ such that,
			\( \pmb{w}_1=\psi_{\bar{u}_1}(w_0,\pmb{x}_0), \pmb{w}_2=\psi_{\bar{u}_2}(w_0,\pmb{x}_0). \)
			The transition $ T^+ $ in these coordinates is hence computed as, 
			\[ \pmb{w}_2 = \psi_{\bar{u}_2} \circ \psi_{-\bar{u}_1}(\pmb{w}_1). \]
			This computation yields,
			\begin{equation}\label{eqn:psisol}
				\begin{aligned}
					w_2 			&=  w_1 + O(w_1^7) \\
					\bar{x}_1 		&= \bar{x}_1 + O(w_1^9) \\
					\bar{h}_{12}	&= \bar{h}_{11} + b_c a_1^{-1/3} \left(\bar{H}^{(8)}(\bar{u}_1) - \bar{H}^{(8)}(\bar{u}_2)\right) w_1^8 + O(w_1^9)\\
					\bar{h}_{22}	&= \bar{h}_{21} - b_c a_2^{-1/3} \left(\bar{H}^{(8)}(\bar{u}_1) - \bar{H}^{(8)}(\bar{u}_2)\right) w_1^8 + O(w_1^9) \\
					\bar{y}			&= \bar{y}_1 + O(w_1^9).				
				\end{aligned}
			\end{equation}
			
			The transition $ T^+:\sigma_1^\nu \to \sigma_2^\nu $ will follow after replacing $ \pmb{w}_1,\pmb{w}_2 $ in \eqref{eqn:psisol} by their respective parameterisation through $ P_{\tilde{z}_2}\circ P_{\tilde{z}_1}^{-1} \circ \Phi^{-1}(\sigma_i) $. Let $ (-u_1,-\nu,x_1,h_{11},h_{21},y_1),(u_2,\nu,x_2,h_{12},h_{22},y_2) $ be the normal form coordinates on $ \sigma_1,\sigma_2 $ respectively. Using the fact that $ w = \tilde{\kappa}^{1/3} = u v^{1/3} $ and substituting into \eqref{eqn:psisol} it follows that
			\[ u_2 = u_1 + O(u_1^7). \]
			It is also evident $ x_2 = x_1 + O(u_1^9),\, y_2 = y_1 + O(u_1^9) $. To get the $ h $ transitions, we need to explicitly compute the parameterisation. The first step is to find the inverse of $ \Phi $ from its definition in \eqref{eqn:normalFormCoords}. We have,
			\[ -\nu = \frac{1}{3}\hat{v}_1(3+ \hat{v}_1^2) + O(\hat{u}_1^4) \implies \hat{v}_1 \sim -\nu + O(\nu^3,u^4). \]
			Then, using that $ (\bar{u},\bar{v}) = P_{\tilde{z}_2}\circ P_{\tilde{z}_1}^{-1} (\hat{u},\hat{v}) = (\hat{v}^{-1}, \hat{u}\hat{v}) $ we obtain
			\[ \bar{u}_1(\nu) = (-\nu)^{-1} + O(\nu^2, u_1^4),\qquad \bar{u}_2(\nu) = (\nu)^{-1} + O(\nu^2, u_1^4). \]
			Again using the normal form coordinates \eqref{eqn:normalFormCoords}, the $ h_1 $ parameterisations are computed as
			\[ \bar{h}_{11} = h_{11} + u_1^8 b_c \frac{216 \nu }{95 a_1^{1/3} } + O(\nu^2,u_1^9),\qquad \bar{h}_{12} = h_{12} - u_1^8 b_c \frac{216 \nu }{95 a_1^{1/3} } + O(\nu^2,u_1^9) . \]
			A similar expression is obtained for $ \bar{h}_{21},\bar{h}_{22} $. Finally, substituting each parameterisation into \eqref{eqn:psisol} the transition map $ T^+:\sigma_1^\nu\to\sigma_2^\nu $ is explicitly computed as
			\begin{equation}
				\begin{aligned}
					u_2 		&=  u_1 + O(u_1^7) \\
					x_2 		&=  x_1 + O(u_1^9) \\
					h_{12}		&= 	h_{11} + b_c a_1^{-1/3} H^8(\nu) u_1^8 + O(u_1^9)\\
					h_{22}		&=  h_{21} - b_c a_2^{-1/3} H^8(\nu) u_1^8 + O(u_1^9) \\
					y_2			&= 	y_1 + O(u_1^9)
				\end{aligned}
			\end{equation}
			where
			\[ H^8(\nu) = \frac{432}{95} \nu + \left(\bar{H}^{(8)}(-\nu^{-1}) - \bar{H}^{(8)}(\nu^{-1})\right) + O(\nu^3) = -24 \cdot 3^{1/6} \sqrt{\pi} \frac{\Gamma\left( -5/6 \right)}{\Gamma\left(2/3\right)} \nu^{8/3} + O(\nu^3). \]
			
		\subsubsection{Gluing Together}
		At last we are in a position to give the asymptotic expansion of the block map $ \pib^+ $. Composing the maps $ \pib^+_\nu = D_2^\nu\circ T^+_\nu D_1^\nu $ and taking the limit $ \nu \to 0 $ gives the result:
		\begin{align*}
			\pib^+_\nu &= D_2^\nu\circ T^+_\nu\circ D_1^\nu(v,x,h_1,h_2,y) \\ 
							&= D_2^\nu \circ T^+_\nu \left(\left(\nu^{-1/3}v^{1/3},x,h_1,h_2,y\right) + O(v^3\ln v)\right) \\
							&= D_2^\nu \left(\nu^{-1/3} v^{1/3}, x, h_1 + \tilde{b}_c a_1^{-1/3} v^{8/3},h_2 + \tilde{b}_c a_2^{-1/3} v^{8/3} , y)+O\left(\nu^{1/3},v^3\ln v\right)\right) \\
							&= \left(\nu\left(\nu^{-1/3} v^{1/3}\right)^3, x, h_1 + \tilde{b}_c a_1^{-1/3} v^{8/3},h_2 + \tilde{b}_c a_2^{-1/3} v^{8/3} , y)+O\left(\nu^{1/3},v^{3}\ln v\right)\right) \\
			\lim_{\nu\to 0} \pib^+_\nu &= \left(v,x, h_1 + \tilde{b}_c a_1^{-1/3} v^{8/3},h_2 + \tilde{b}_c a_2^{-1/3} v^{8/3} , y\right) + O(v^3\ln v),
		\end{align*}
		where $ \tilde{b}_c = -24 b_c \cdot 3^{1/6} \sqrt{\pi} \frac{\Gamma\left( -5/6 \right)}{\Gamma\left(2/3\right)} $.
		With this calculation, and noting that $ \tilde{b}_c $ is a strictly positive function of the masses, we have shown the main theorem.
		
		\begin{thm}\label{thm:C83Regularisable}
			For any choice of masses, the simultaneous binary collision is precisely $ C^{8/3} $-regularisable in the collinear 4-body problem.
		\end{thm}
		
\section{Concluding Remarks}
	A new proof of the $ C^{8/3} $-regularity of simultaneous binary collisions in the collinear $ 4 $-body problem has been given. In the process, new results about the problem have been shown. The $ C^{8/3} $ differentiability is now known to hold for any choice of masses and for almost all directional derivatives of $ \pib $. The exception is when $ v = 0 $, that is, a derivative is taken in the direction along the set of collision orbits $ \CE^+ $. This last result was known to Elbialy \cite{Elbialy1993planar}. Moreover, a heuristic of Martinez and Sim\'{o} on the crucial role of the first coupling term $ b_c z_1^4 z_2^4 $  was proved by explicitly computing the asymptotics of the block map.
	
	We sought a more geometric proof to that in \cite{Martinez1999}. It is now clear that is not possible to construct a set of integrals local to the set of singularities. Another geometric notion essential to the proof was an investigation of hyperbolic transitions near normally hyperbolic manifolds of fixed points. From this work, the following can now be concluded about the mysterious $ 8/3 $ differentiability:
	\begin{itemize}
		\item The $ 1/3 $ results from the ratio of hyperbolicity of the two normally hyperbolic manifolds $ \NHIM^+,\NHIM^- $.
		\item The $ 8 $ results from the inability to construct an invariant foliation normal to the set of simultaneous binary collisions $ \sbc $ at order $ 8 $ in the intrinsic energies $ h_i $.
	\end{itemize}

	A physical interpretation of the non-smoothness is as follows.
	Near the simultaneous binary collision the energy of the individual binaries before and after collision is a non-smooth function of a measure of their difference in coordinates. Expressing the final variable $v$ on the section $u = \pm 1$ in terms of the original variables $Q_i$ gives to leading order 
	\[
		v = t_1^{1/3} - t_2^{1/3}, \quad
		t_i^2 = Q_i^3 \frac{ M_i }{ k_i } \,.
	\]
	Here $ t_i $ is just an abbreviation, but it has units of time and can be considered as the leading order term in the solution $ Q_i(t_i) $ near collision $ Q_i = 0 $.	
	
	Whilst the explanation of the finite differentiability given is certainly succinct, the computation required to prove the theorem is overly cumbersome. Mathematica was essential in computing the several normal forms and the variational equations. However, the approach is very general and can theoretically be extended to other problems, for example, the planar problem or $ n > 4 $.
	
	Future work would do well to investigate the connection between the inability to foliate and the finite differentiability of the block map. Perhaps by investigating regularisation near arbitrary manifolds of fixed points, a deeper theory could be formulated, making the large computations of this paper an unnecessary burden.

	\bibliography{SBC_Collinear}
	\bibliographystyle{hplain}

\end{document}